\documentclass[a4paper,12pt,thmsa]{amsart}
\usepackage[a4paper,marginratio={1:1},scale={0.72,0.74},footskip=7mm,headsep=10mm]{geometry}

\usepackage{amsfonts}
\usepackage{amssymb,amsmath,latexsym}
\usepackage[dvips]{graphics}
\usepackage{graphicx,subfigure}
\usepackage[dvips]{color}
\usepackage[active]{srcltx}
\usepackage{amsmath}
\usepackage{amsfonts}
\usepackage{amssymb}
\usepackage{psfrag}
\usepackage{color}
\usepackage{url}
\usepackage{amsthm}
\usepackage{array}
\usepackage{pst-tree}
\usepackage{lscape}
\usepackage{dsfont}
\usepackage[utf8]{inputenc}							
\usepackage[T1]{fontenc}   
\usepackage{pstricks,pstricks-add}
\usepackage{mathrsfs}  
\usepackage{bbm}
\usepackage{tikz}

\newtheorem{theorem}{Theorem}[section]
\newtheorem{proposition}{Proposition}[section]

\newtheorem{corollary}{Corollary}[section]

\newtheorem{lemma}{Lemma}[section]

\newcommand{\ed}{\stackrel{\mbox{\tiny $(d)$}}{=}}

\newcommand{\R}{\mathbb{R}}
\newcommand{\N}{\mathbb{N}}

\def\e{\mathbb{E}}
\def\p{\mathbb{P}}
\newcommand{\ind}{\mbox{\rm 1\hspace{-0.04in}I}}

\title[Creeping of L\'evy processes through curves]
{Creeping of L\'evy processes through curves}

\author{Lo\"ic Chaumont}

\address{Lo\"ic Chaumont, Univ Angers, CNRS, LAREMA, SFR MATHSTIC, F-49000 Angers, France}

\email{loic.chaumont@univ-angers.fr}

\author{Thomas Pellas}

\address{Thomas Pellas, Univ Angers, CNRS, LAREMA, SFR MATHSTIC, F-49000 Angers, France}

\email{thomas.pellas@univ-angers.fr}



\date{\today}

\begin{document}

\begin{abstract} A L\'evy process is said to creep through a curve if, at its first passage time 
across this curve, the process reaches it with positive probability. 
We first study this property for bivariate subordinators. Given the graph $\{(t,f(t)):t\ge0\}$ 
of any continuous, non increasing function $f$ such that $f(0)>0$, we give an expression of 
the probability that a bivariate subordinator $(Y,Z)$ issued from 0 creeps through this graph in 
terms of its renewal function and the drifts of the components $Y$ and $Z$.
We apply this result to the creeping probability of any real L\'evy process through the graph of 
any continuous, non increasing function at a time where the process also reaches its past supremum.  
This probability involves the density of the renewal function of the bivariate upward ladder process 
as well as its drift coefficients. We also investigate the case of L\'evy processes conditioned to 
stay positive creeping at their last passage time below the graph of a function. Then we provide 
some examples and we give an application to the probability of creeping through fixed levels by 
stable Ornstein-Uhlenbeck processes. We also raise a couple of open questions along the text.
\end{abstract}

\maketitle

\section{Introduction}\label{int}

A real L\'evy process $X$ starting from 0 under $\p$ is said to creep through the level $x>0$ if 
\begin{equation}\label{2253}
\mathbb{P}(X_{\tau_x^+}=x)>0\,,
\end{equation}
where $\tau_x^+=\inf\{t:X_t>x\}$. The first study of this property is due to Millar \cite{mi} who,
according to his own words, gives a 'reasonably complete solution' of the creeping problem, that is conditions
bearing on the characteristics of $X$ for (\ref{2253}) to hold. It is proved in \cite{mi} that if (\ref{2253}) 
holds for some $x>0$, then it holds
for all $x>0$. We will then simply say that $X$ creeps upward. According to the same paper, when $X$ has
bounded variation, it creeps upward if and only if it has positive drift. When $X$ has unbounded variation,
it creeps upward if it has a Brownian part or if its L\'evy measure $\pi$ satisfies $\int_0^1xd\pi(x)<\infty$.
It does not creep upward if it has no Brownian part and $\int_{-1}^0|x|d\pi(x)<\infty$. When $X$ has no Brownian
part and $\int_0^1xd\pi(x)=\int_{-1}^0|x|d\pi(x)=\infty$, both possibilities can occur and an analytic criterion 
in terms of the characteristic exponent of $X$ is given. As pointed out by Vigon \cite{vi}, this criterion has 
the disadvantage of involving both the characteristic exponent of $X$ and its L\'evy measure and we rarely know 
both at the same time. This gap is filled in \cite{vi} where an integral test only bearing on the L\'evy measure 
is established for a L\'evy process with unbounded variation and no Brownian part to creep upward. Let us also 
mention the work of Rogers \cite{ro} who deduced some results of Millar. 
It seems that it was in the latter article that the term 'creeping' first appeared.\\

According to a remark first made in \cite{mi} and then used in \cite{vi}, the upward ladder 
height process $H$ of $X$ satisfies $H_{S_x^+}=X_{\tau_x^+}$, $x>0$, where $S_x^+=\inf\{t:H_t>x\}$, 
so that creeping of general L\'evy processes can in fact be reduced to that of subordinators. More specifically, 
$X$ creeps upward if and only if $H$ creeps upward, which holds if and only if the drift $d_H$ of $H$ is 
positive according to Neveu \cite{ne} and Kesten \cite{ke}. These authors also proved that in this case, the 
renewal measure of the ladder height process admits a continuous and bounded density $u(x)$ and that the 
creeping probability is then given by 
\begin{equation}\label{7433}
\mathbb{P}(X_{\tau_x^+}=x)=d_Hu(x)\,.
\end{equation}
So the whole creeping problem for a L\'evy process $X$ boils down to giving conditions on its characteristics
for $d_H$ to be positive. Let us also mention a more recent result in this direction by Griffin and Maller 
\cite{gm} who studied the law of the creeping time, that is the law of the first passage time conditionally on 
the creeping event.\\
 
In the present work, we are interested in the probability for L\'evy processes to creep through the graph of   
continuous non increasing functions. We first consider the case of a bivariate subordinator $(Y,Z)$, that 
is a two dimensional L\'evy process issued from 0, whose both coordinates are non decreasing. Let 
$f:(0,\infty)\mapsto(0,\infty)$ be a continuous non increasing function. Our first main result gives an 
expression of the probability that the path $\{(Y_t,Z_t):t>0\}$ crosses the graph $\{(t,f(t)):t>0\}$ in a 
continuous way, that is 
\[\p(Z_{S}=f(Y_S),\,u_0<Y_{S}<u_1),\;\;\;\mbox{where}\;\;\;S=\inf\{t:Z_t>f(Y_t)\},\] 
for all $0\le u_0<u_1\le\infty$. This expression involves the drifts of $Y$ and $Z$ and the renewal function 
of the process $(Y,Z)$, see Theorem \ref{9049}. In particular, for $(Y,Z)$ to creep through the graph of $f$ 
with positive probability, it is necessary that at least one of the drifts of $Y$ and 
$Z$ is positive. This result will then be applied in Theorem \ref{3636} to the probability for a real 
L\'evy process $X$ to creep at its supremum through the graph of $f$. More formally, define the first passage 
time above the curve of $f$ by
\[T_f=\inf \lbrace t>0:X_t>f(t) \rbrace \,.\]
Then we give an explicit expression of the probability 
$\p(X_{T_f}=\overline{X}_{T_f}=f({T_f}),\,t_0<T_f<t_1)$, for all $0\le t_0<t_1\le \infty$, where 
$\overline{X}_{T_f}=\sup_{t\le T_f}X_t$. Our result shows that this probability is composed of two terms. 
The first one represents the contribution of the accumulation of jumps of the ladder height process. This 
shows that the function acts as a fixed (horizontal) barrier through which the process creeps upward. Then 
the second term is due to the accumulation of jumps of the ladder time process. When its drift is positive, 
that is when 0 is not regular for $(-\infty,0)$, this term represents a kind of creeping forward, as if the 
function acted as a vertical barrier. Note that our result actually provides the law of $T_f$ conditionally 
on the event $\{X_{T_f}=\overline{X}_{T_f}=f({T_f})\}$. Our only assumption is that the renewal measure of 
the bivariate upward ladder process has a continuous density. As proved in the present work, it is ensured 
whenever $X_t$ has a bounded density for each $t>0$. The latter result can then be applied to the creeping 
of the L\'evy process $X$ conditioned to stay positive. Thanks to a time reversal argument, we obtain in 
Theorem \ref{3737} the probability for this process to creep at its last passage time below the function 
$f$.\\

The problem of creeping through fixed levels by Ornstein-Uhlenbeck processes is at the origin of our work. 
Indeed, the process $Z_t=z+X_t-\gamma \int_0^t Z_s ds$, $t\ge0$, $\gamma>0$, $z<0$ creeps through $y\in(z,0)$ 
if and only if $X$ creeps through the adapted, continuous decreasing functional 
$t\mapsto y-z+\gamma\int_0^t Z_s ds$. This problem is still an open question. However, when $X$ is stable, 
it can be reduced to the creeping of $X$ through a deterministic function. Then we prove, as an application 
of our main result, that Ornstein-Uhlenbeck processes driven by stable processes with index less than 1 always 
creep through some fixed levels.\\ 

We present our main results in the next section. Theorems \ref{9049}, \ref{3636} and \ref{3737} will be 
proved in Sections \ref{proof}, \ref{proof2} and  \ref{proof3}, respectively. We give some examples for which 
the creeping probabilities can be computed explicitly in Section \ref{examples}, and in Section \ref{OU} we 
present the above-mentioned application of our results to the creeping property of stable Ornstein-Uhlenbeck 
processes.

\section{Main results}\label{main}

\subsection{Creeping of bivariate subordinators through curves}\label{3025}
A bivariate subordinator $(Y,Z)$ is a possibly killed two dimensional L\'evy process starting from 0, 
whose coordinates $Y$ and $Z$ are non decreasing. We set $Y_t=Z_t=\infty$, for $t\in[\zeta,\infty]$, where $\zeta$ 
is the lifetime of $(Y,Z)$. In what follows, the drift of a (univariate) subordinator $Y$ 
will be denoted by $d_Y$. Moreover, a measure on some subset of $\mathbb{R}^d$ is said to be absolutely 
continuous if it is absolutely continuous  with respect to the Lebesgue measure on this subset.\\ 

Our first theorem concerns the creeping of a bivariate subordinator $(Y,Z)$ through the curve defined by the graph 
of a continuous non increasing function $f$. More specifically, we are interested in the probability that the path 
$\{(Y_t,Z_t):t\ge0\}$ crosses continuously the graph $\{(u,f(u)):u\ge0\}$ of $f$. This event can be written as
$\{Z_{S}=f(Y_S)\}$, where $S=\inf\{t:Z_t>f(Y_t)\}$, see Figure \ref{sub_biv}. When it holds,
we will say that $(Y,Z)$ creeps through the graph of $f$.

\begin{theorem}\label{9049}
Let $(Y,Z)$ be a bivariate subordinator with semigroup $\p(Y_t\in dy,\,Z_t\in dz)=p_t(dy,dz)$, $t>0$, 
$y,z\in[0,\infty)$. Assume that the renewal measure $v(dy,dz):=\int_0^{+\infty}p_t(dy,dz)\,dt$ is absolutely 
continuous on $(0,\infty)^2$ and that its density $(y,z)\mapsto v(y,z)$ is continuous on $(0,\infty)^2$. 
Let $f:(0,\infty)\rightarrow(0,\infty)$ be a continuous, non increasing function. Set 
$f(0)=\lim_{t\rightarrow0+}f(t)$, $f(\infty)=\lim_{t\rightarrow\infty}f(t)$ and define,
\[S=\inf\{t:Z_t>f(Y_t)\}\,.\] 
\begin{itemize}
\item[$1.$] Then almost surely, $0<S<\infty$, $Y_{S-}<\infty$ and $Z_{S-}<\infty$. Moreover, the process $(Y,Z)$ 
can neither jump on the graph $\{(u,f(u)):u>0\}$ nor jump from this graph, that is 
\[\p(Z_{S-}=f(Y_{S-}),\Delta(Y,Z)_{S}\neq0)=\p(Z_{S}=f(Y_{S}),\Delta(Y,Z)_{S}\neq0)=0\,,\]
where $\Delta(Y,Z)_{S}=(Y,Z)_{S}-(Y,Z)_{S-}$. 
\item[$2.$] Let $u_0,u_1$ be such that $0\le u_0<u_1\le \infty$. Then the creeping probability of $(Y,Z)$ 
through the graph of $f$ is given by, 
\begin{equation}\label{7229}
\p(Z_{S}=f(Y_S),\,u_0<Y_{S}<u_1)=
d_Z\int_{u_0}^{u_1}v(u,f(u)) \, du-d_Y\int_{u_0}^{u_1}v(u,f(u))\, df(u)\,.
\end{equation}
\item[$3.$] Assume that the measure $p_t(dy,dz)$ is absolutely continuous on $(0,\infty)^2$ for all $t>0$ 
and that the densities $(t,y,z)\mapsto p_t(y,z)$ are continuous on $(0,\infty)^3$. Then for all
$t_0,t_1,u_0,u_1$ such that $0\le t_0<t_1\le \infty$ and $0\le u_0<u_1\le \infty$, 
\begin{eqnarray}
&&\p(Z_{S}=f(Y_S),\,u_0<Y_{S}<u_1,\,t_0<S<t_1)=\nonumber\\
&&\qquad\qquad d_Z\int_{u_0}^{u_1}\int_{t_0}^{t_1}p_t(u,f(u))\,dt \, du-d_Y\int_{u_0}^{u_1}\label{6366}
\int_{t_0}^{t_1}p_t(u,f(u))\,dt\, df(u)\,.
\end{eqnarray}
\end{itemize}
\end{theorem}

\noindent Let us mention that an intermediary result between those of part 2.~and part 3.~of Theorem 
\ref{9049} can be stated as follows: 
\begin{itemize}
\item[$2'.$] Assume that for some $0\le t_0<t_1\le\infty$, the measure $\int_{t_0}^{t_1}p_t(dy,dz)\,dt$ is 
absolutely continuous on $(0,\infty)^2$, with density $\bar{v}(y,z)$ and that $\bar{v}$ is continuous on 
$(0,\infty)^2$. Then for all $u_0,u_1$ such that $0\le u_0<u_1\le \infty$,
\begin{eqnarray*}
&&\p(Z_{S}=f(Y_S),\,u_0<Y_{S}<u_1,\,t_0<S<t_1)=\nonumber\\
&&\qquad\qquad d_Z\int_{u_0}^{u_1}\bar{v}(u,f(u)) \, du-d_Y\int_{u_0}^{u_1} 
\bar{v}(u,f(u))\, df(u)\,.
\end{eqnarray*}
\end{itemize}
The proof of $2'.$~is quite similar to this of part 3.~of Theorem \ref{9049}, see subsections
\ref{8590} and \ref{9791}, so it is omitted.\\ 

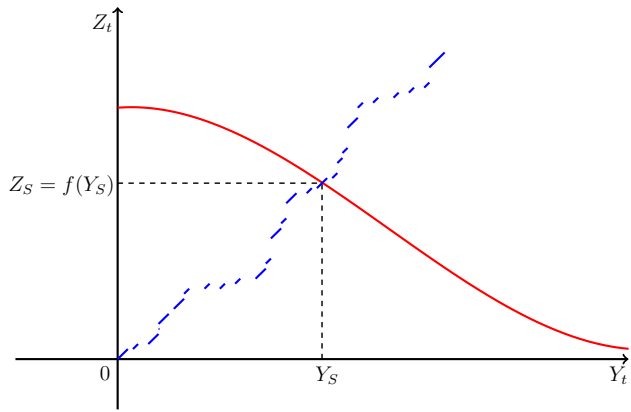
\begin{figure}[!h]
\centering
\resizebox{9.5cm}{6cm}{
 \begin{tikzpicture}[xscale=1,yscale=1]
			\newcommand{\xmin}{-4}				
			\newcommand{\xmax}{10}
			\newcommand{\ymin}{-2}
			\newcommand{\ymax}{7}
			\clip (\xmin,\ymin) rectangle (\xmax,\ymax);
			\draw [very thick, ->](-2,0)--(\xmax,0);
			\draw [very thick, ->](0,-1)--(0,\ymax);
			\draw (0,0) node [below left] {$0$};
			\draw(4.1,0)node[below]{$Y_S$};
			\draw(9.8,0)node[below]{$Y_t$};
			\draw(-0.3,7)node[below]{$Z_t$};
			\draw(-1.1,3.8)node[below]{$Z_S=f(Y_S)$};
			\draw[dashed, thick](4,0)--(4,3.5);
			\draw[dashed, thick](0,3.5)--(4,3.5);
			\draw[red, domain=0:\xmax, samples=200,very thick] plot (\x, {0.0095*\x*\x*\x-0.151*\x*\x+0.08*\x+5}); 
			\draw[blue, very thick] plot file {coord_biv.txt};
	\end{tikzpicture}}
	\caption{A sample path of $(Y,Z)$ creeping through the graph of $f$}
	\label{sub_biv}
\end{figure}

Let us now focus on two direct applications of Theorem \ref{9049}. The first one gives the probability 
for the Euclidian norm $\|(Y,Z)\|=\sqrt{Y^2+Z^2}$ of a bivariate subordinator to creep through a fixed level, 
that is $\p(\|(Y,Z)_{U_a}\|=a)$, for $a>0$ and $U_a=\inf\{t:\|(Y,Z)_t\|>a\}$.
This is simply obtained from Theorem \ref{9049} by choosing $f(y)=\sqrt{a^2-y^2}$, so that $S=U_a$.
\begin{corollary}\label{2524}
Let $(Y,Z)$ be a bivariate subordinator. Then with the same notation and assumptions than in 
Theorem $\ref{9049}$,
\begin{equation}\label{2611}
\p(\|(Y,Z)\|_{U_a}=a)=d_Z\int_{0}^{a}v(u,\sqrt{a^2-u^2}) \, du+d_Y\int_{0}^{a}
\frac{u v(u,\sqrt{a^2-u^2})}{\sqrt{a^2-u^2}}\, du\,.
\end{equation}
\end{corollary}

The second direct application concerns the probability for a (one dimensional) subordinator to creep through 
the graph of a continuous non increasing function. Let $X$ be a subordinator whose semigroup has continuous 
densities $(t,x)\mapsto p_t(x)$ on $(0,\infty)^2$ and let us apply identity 
(\ref{7229}) to the bivariate subordinator $(Y_t,Z_t)=(t,X_t)$. 
The renewal measure of $(Y,Z)$ is then $v(t,x)\,dt\,dx=p_t(x)\,dt\,dx$. Hence it satisfies the assumption of
part 2.~of Theorem \ref{9049}. Moreover, since $Y_t=t$, one has $Y_S=\inf \lbrace t: X_t>f(t) \rbrace$. This 
leads to the following corollary. 
\begin{corollary}\label{2526}
Assume that $X$ is a subordinator whose 
semigroup $\p(X_t\in dx)$ is absolutely continuous on $(0,\infty)$ for each $t>0$, with continuous densities 
$(t,x)\mapsto p_t(x)$ on $(0,\infty)^2$. Set $T_f=\inf \lbrace t>0: X_t>f(t) \rbrace$, where $f$ is as in the 
statement of Theorem $\ref{9049}$. Then for all $t_0$ and $t_1$ such that $0 \le t_0<t_1\le \infty$,
\begin{eqnarray}
\mathbb{P}(X_{T_f}=f({T_f}),t_0<{T_f}<t_1)
= d_X\int_{t_0}^{t_1}p_u(f(u))\,du-\int_{t_0}^{t_1}p_u(f(u)) \, df(u)\,.\label{7372}
\end{eqnarray}
\end{corollary}
\noindent Let us mention that the creeping probability for subordinators with no drift has been determined in 
\cite{chi} when $f$ is differentiable. We will actually extend (\ref{7372}) to the creeping probability of any 
L\'evy process at its supremum in the next section.\\

\noindent Note that if $f$ is decreasing, then (\ref{7229}) can also be written as
\[\p(Z_{S}=f(Y_S),\,u_0<Y_{S}<u_1)=
d_Y\int_{u_0}^{u_1}v(f^{-1}(u),u) \, du-d_Z\int_{u_0}^{u_1}v(f^{-1}(u),u)\, df^{-1}(u)\,,\] 
and this applies to (\ref{2611}) and (\ref{7372}).\\ 

The question of finding necessary and sufficient conditions for a bivariate subordinator $(Y,Z)$ to creep through
a given continuous, non increasing function is still open. In particular, nothing is known when the renewal measure 
$v(dy,dz)$ is not absolutely continuous. Looking at (\ref{7229}) in Theorem \ref{9049}, one is tempted to think that 
when $d_Y>0$ and $d_Z>0$, the probability $\p(Z_{S}=f(Y_S))$ is always positive and when $d_Y=0$ and $d_Z=0$, this 
probability is always equal to 0. Let us also mention that the absolute continuity condition of Theorem \ref{9049} 
is discussed in Subsection \ref{3331}.\\

It seems hardly possible to obtain an expression of the creeping probability when $f$ is a general continuous 
function. When $f$ is continuous and non decreasing we can still obtain an upper bound for the creeping probability
as the following proposition shows. 

\begin{proposition}\label{8983}  
Let $(Y,Z)$ be a bivariate subordinator which satisfies the same assumptions as in part $2$.~of Theorem $\ref{9049}$.   
Keep the same notation as in this theorem. Let $f:(0,\infty)\rightarrow(0,\infty)$ be a continuous non decreasing 
function such that $\lim_{t\rightarrow0+}f(t)>0$. Set  $f(0)=\lim_{t\rightarrow0+}f(t)$, 
$f(\infty)=\lim_{t\rightarrow\infty}f(t)$ and define $S=\inf\{t:Z_t>f(Y_t)\}$.
\begin{itemize}
\item[$1.$] Then $\p(S>0)=1$ and the process $(Y,Z)$ can neither jump on the graph $\{(u,f(u)):u>0\}$ nor jump from 
this graph, that is 
\[\p(Z_{S-}=f(Y_{S-}),\Delta(Y,Z)_{S}\neq0,Y_{S}<\infty)=\p(Z_{S}=f(Y_{S}),\Delta(Y,Z)_{S}\neq0,Y_{S}<\infty)=0\,.\]
\item[$2.$] Let $u_0,u_1$ be such that $0\le u_0<u_1\le \infty$. Then the creeping probability of $(Y,Z)$ 
through the graph of $f$ is bounded from above as follows, 
\begin{equation}\label{6633}
\p(Z_{S}=f(Y_S),\,u_0<Y_{S}<u_1)\le d_Z\int_{u_0}^{u_1}v(u,f(u))\,du\,.
\end{equation}
\end{itemize}
\end{proposition}

\noindent Proposition \ref{8983} implies that when $d_Z=0$, the process $(Y,Z)$ never creeps through the graph of 
any continuous non decreasing function. This result and its consequence in Proposition \ref{5838} will be useful for
our application to Ornstein-Uhlenbeck processes in Section \ref{OU}.

\subsection{Creeping of real L\'evy processes at their supremum}\label{6525}
Let us now consider a non killed real L\'evy process $X$. We will always assume in the sequel that $X_0=0$, a.s.
Since our interest lies in creeping upward of $X$, it is natural to assume throughout this paper that $-X$ is not a 
subordinator. Moreover the case of subordinators has already been dealt with in Corollary \ref{2526}. Therefore in this 
subsection, we will assume that $|X|$ is not a subordinator. We write $\overline{X}$ for the supremum process, that is,
\[\overline{X}_{t}=\sup_{s\le t}X_{s}\,,\;\;\;t\ge0\,.\]
Then let us first recall a few basic definitions. It is well known that the reflected process $\overline{X}-X$ 
is strongly Markovian under $\p$. 
Let $L$ be the local time at 0 of $\overline{X}-X$ (a proper definition is given in Subsection \ref{5399}) 
and denote by $(\tau,H)$ the upward ladder process of $X$, that is the bivariate L\'evy process whose 
coordinates are the following (possibly killed) subordinators:
\[\tau_t=\inf\{s:L_s>t\}\;\;\;\mbox{and}\;\;\;H_t=X_{\tau_t}\,,\;\;t\ge0\,,\]
where $\tau_t=H_t=\infty$, for $t\ge L_\infty$. The process $\tau$ (resp. $H$) is called the upward
ladder time (resp.~height) process of $X$. We will denote by $d_\tau$ and $d_H$ their respective drift 
coefficients. Let us introduce the renewal measure $U$ on $[0,\infty)^2$ of the ladder process 
$(\tau,H)$, that is,
\[U(dt,dh)=\int_0^\infty\p(\tau_u\in dt,\,H_u\in dh)\,du\,.\]
We specify that when the measure $U(dt,dh)$ is absolutely continuous on $(0,\infty)^2$, its density will 
be denoted by $q_t^*(h)$. This notation may seem unnatural but it comes from the fact that this density 
corresponds to the entrance law of the reflected excursions, see Subsection \ref{5399}, which is thus noted 
in older references, see \cite{ch} and \cite{cm} for instance.\\  

Throughout this paper, for any function $f:(0,\infty)\rightarrow\mathbb{R}$, we set 
\[T_f:=\inf\{t>0:X_t>f(t)\}\,.\]
When a L\'evy process $X$ satisfies $\p(X_{T_f}=\overline{X}_{T_f}=f({T_f}),\,t_0<{T_f}<t_1)>0$, we will 
say that $X$ creeps at its supremum through the function $f$ over the interval $(t_0,t_1)$.
Omission of $(t_0,t_1)$ will simply mean that X creeps over the whole half line $(0,\infty)$.
The following result gives an expression of the probability of this event when $f$ is a continuous, non 
increasing function.

\begin{theorem}\label{3636}
Let $X$ be a real L\'evy process and $f:(0,\infty)\rightarrow(0,\infty)$ be a continuous, non increasing 
function.
\begin{itemize}
\item[$1.$] Then $\p(T_f>0)=1$ and the process $X$ can neither jump on the graph $\{(u,f(u)):u>0\}$ 
nor jump from this graph, that is 
\[\p(X_{T_f-}=f(T_f)<X_{T_f},\,T_f<\infty)=\p(X_{T_f-}<X_{T_f}=f(T_f),\,T_f<\infty)=0\,.\]
\item[$2.$] Assume that the renewal measure $U(dt,dx)$ of the ladder process $(\tau,H)$ has a continuous 
density $(t,x)\mapsto q_t^*(x)$ on $(0,\infty)^2$. Then $X$ creeps at its supremum through $f$ if and only
if its ladder process $(\tau,H)$ creeps through the graph of $f$. Moreover, for all $t_0$ and $t_1$ such that 
$0 \le t_0<t_1\le \infty$, 
\begin{equation}\label{8225}
\p(X_{T_f}=\overline{X}_{T_f}=f({T_f}),\,t_0<{T_f}<t_1)=
d_H\int_{t_0}^{t_1}q^*_u(f(u)) \, du-d_\tau\int_{t_0}^{t_1}q^*_u(f(u))\, df(u)\,.
\end{equation}
\end{itemize}
\end{theorem}

The creeping probability of part 2.~of Theorem \ref{3636} reveals two types of creeping at the supremum,  
each of which corresponds to one of the two terms involved. When $d_H>0$, the process creeps upward 
either continuously (in absence of positive jumps) or through an accumulation of jumps of $H$ as in the 
case of a fixed barrier. In the case where $d_\tau>0$ (which is equivalent to the fact that 
0 is not regular for $(-\infty,0)$), another type of creeping occurs. Then the process creeps forward through 
an accumulation of jumps of $\tau$ provided $f$ decreases at the creeping time. Recall that the later jumps 
correspond to the lengths of the excursions of $\overline{X}-X$. This new kind of creeping happens as if $f$ 
were acting as a vertical barrier.\\ 

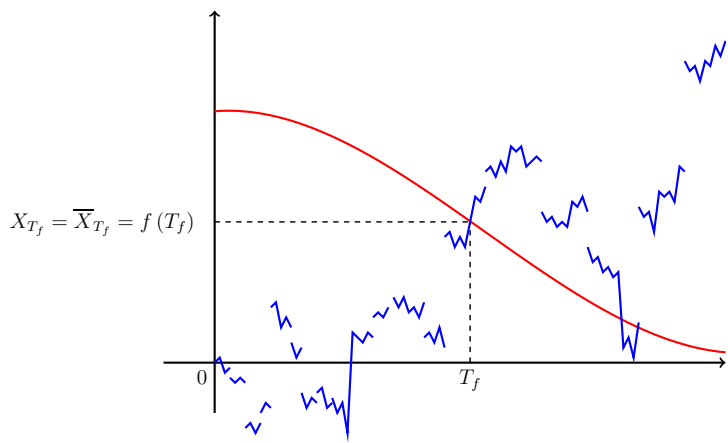
\begin{figure}[!h]
\centering
\resizebox{9.5cm}{6cm}{
\begin{tikzpicture}[xscale=1,yscale=1]
			\newcommand{\xmin}{-4}				
			\newcommand{\xmax}{10}
			\newcommand{\ymin}{-2}
			\newcommand{\ymax}{7}
			\clip (\xmin,\ymin) rectangle (\xmax,\ymax);
			\draw [very thick, ->](-1,0)--(\xmax,0);
			\draw [very thick, ->](0,-1)--(0,\ymax);
			\draw (0,0) node [below left] {$0$};
			\draw(5,0)node[below]{$T_f$};
			\draw[dashed, thick](0,2.8)--(5,2.8);
			\draw[dashed, thick](5,0)--(5,2.8);
			\draw(-2.2,3.2)node[below]{$X_{T_f}=\overline{X}_{T_f}=f\left({T_f}\right)$};
			
			\draw[red,domain=0:\xmax, samples=200,very thick] plot (\x, {0.0095*\x*\x*\x-0.151*\x*\x+0.08*\x+5}); 
			\draw[blue, very thick] plot file {coord_lev.txt};		
	\end{tikzpicture}}
	\caption{A sample path of $X$ creeping at its supremum through $f$}
	\label{X_creeps}
\end{figure}

We point out that the assumption of absolute continuity in part 2.~of the above result
is not as strong as it may appear. Indeed, we will see in Proposition \ref{1753}, 
that it is satisfied whenever the transition semigroup of $X$ admits densities, $x \mapsto p_t(x)$, $x \in \R$, 
which are bounded for all $t>0$, and for all $c\ge0$, the process $(|X_t-ct|,\,t\ge0)$ 
is not a subordinator. Moreover, from the same proposition, $q_t^*$ is positive on $(0,\infty)$, for all $t>0$, which 
ensures the positivity of the creeping probability (\ref{8225}) whenever $d_H>0$ or $d_\tau>0$ and $f$ is decreasing. 
On the other hand, when the assumption of Theorem \ref{3636} is not satisfied, we can still expect that the process 
creeps at its supremum through $f$ with positive probability, but finding necessary and sufficient conditions for this to 
hold is an open question which is closely related to the question raised by Theorem \ref{9049}, see the discussion after 
Corollary \ref{2526}.\\

Recall from $(3.3)$ in \cite{cm} that the function $x\mapsto u(x)=\int_0^\infty q_t^*(x)\,dt$
is the potential density of $H$ $($this is true by definition when $X$ is a subordinator since $q_t^*(x)=p_t(x))$. 
Hence by taking $f\equiv x$ in the statement of the above theorem, as well as in Corollary \ref{2526}, we recover 
the classical creeping result for L\'evy processes recalled in the introduction.\\

When $\p(X_{T_f}=f({T_f}),\,t_0<{T_f}<t_1)>0$ we will say that $X$ creeps through the function $f$
(over the interval $(t_0,t_1)$). 
Finding an expression for the probability $\p(X_{T_f}=f({T_f}),\,t_0<{T_f}<t_1)$ seems more complicated 
than for $\p(X_{T_f}=\overline{X}_{T_f}=f({T_f}),\,t_0<{T_f}<t_1)$. Indeed, in the second case, we only 
need to ensure that $X$ stays below its past supremum before time $T_f$, whereas in the first case the 
condition is that the whole path of $X$ stays below the curve of $f$ before time $T_f$. Therefore, the 
expression of this probability must strongly depend on the behaviour of whole paths of $X$ with respect to 
the curve of $f$ before time $T_f$ and it seems hardly possible to perform this computation for a 
general function. However the example below and Corollary \ref{9445} give conditions for $X$ to creep through
a function $f$ in some particular cases.\\ 

Let $X$ be the L\'evy process $X_t=S_t-at$, $t\ge0$, where $S$ is a subordinator without drift and $a>0$. 
Then 0 is not regular for $(0,\infty)$ and it is regular for $(-\infty,0)$, so that $d_H=d_\tau=0$ (see 
Subsection \ref{fluctuation}) and from Theorem \ref{3636}, the process does not creep at its supremum through 
any continuous non increasing function. However, it does creep through the function $f(t)=1-t$, since for any 
$a<1$, the subordinator with positive drift $S_t+(1-a)t$ creeps through the level 1. This simply means that 
the creeping time of $X$ through $f$ is a.s.~not a time at which $X$ reaches its past supremum. From this 
example, we can build a more general result regarding the probability to creep through a function (not 
necessarily at the supremum) of L\'evy processes with bounded variation.

\begin{corollary}\label{9445} Let $X$ be a L\'evy process with bounded variation and nonnegative drift.
Assume that for all $c\ge0$, the process $(|X_t-ct|,\,t\ge0)$ is not a subordinator and  
that for all $t>0$, the distribution $\p(X_t\in dx)$ is absolutely continuous on $\mathbb{R}$ 
with a bounded density. Let $f:(0,\infty)\rightarrow\mathbb{R}$ be a function such that for some 
$0\le t_0<t_1\le\infty$ and $a>0$, $f(t)+at$ is positive, continuous and non increasing over the interval 
$(t_0,t_1)$. Then $X$ creeps through $f$ over the interval $(t_0,t_1)$, that is 
\[\p(X_{T_f}=f({T_f}),\,t_0<{T_f}<t_1)>0\,.\]
\end{corollary}

\noindent The above corollary will be useful for our application to the creeping of Ornstein-Uhlenbeck 
processes in Section \ref{OU}.\\

From Theorem \ref{3636}, for $X$ to creep at its supremum through $f$, it is necessary that $d_H>0$ or 
$d_\tau>0$. In particular, if $0$ is not regular for $[0,\infty)$, then $d_\tau=d_H=0$ (see Subsection 
\ref{fluctuation}) and $X$ never creeps at its supremum through a continuous non increasing function. 
On the other hand,  when $X$ has unbounded variation, $d_\tau=0$ and hence the process can creep at its 
supremum through $f$ only if $d_H>0$. We conjecture that when $X$ has unbounded variation and $d_H=0$, 
then $X$ never creeps through any continuous, non increasing function. 
Indeed, from the results of this section it is reasonable to think that the only chance for a L\'evy process 
$X$ to creep through such a function is that either $d_H>0$ or $X$ has bounded variation.\\

We end this subsection with the case of a continuous non decreasing function.

\begin{proposition}\label{5838}
Let $X$ be a real L\'evy process and $f:(0,\infty)\rightarrow(0,\infty)$ be a continuous, non decreasing 
function  such that $\lim_{t\rightarrow0+}f(t)>0$.  
\begin{itemize}
\item[$1.$] Then $\p(T_f>0)=1$ and the process $X$ can neither jump on the graph $\{(u,f(u)):u>0\}$ nor jump 
from this graph, that is 
\[\p(X_{T_f-}=f(T_f)<X_{T_f},\,T_f<\infty)=\p(X_{T_f-}<X_{T_f}=f(T_f),\,T_f<\infty)=0\,.\]
\item[$2.$] Assume that the renewal measure $U(dt,dx)$ of the ladder process $(\tau,H)$ has a continuous 
density $(t,x)\mapsto q_t^*(x)$ on $(0,\infty)^2$. Then for all $t_0$ and $t_1$ such that 
$0 \le t_0<t_1\le \infty$, 
\begin{eqnarray*}
\p(X_{T_f}=f({T_f}),\,t_0<{T_f}<t_1)\le d_H\int_{t_0}^{t_1}q^*_u(f(u)) \, du\,.
\end{eqnarray*}
\end{itemize}
\end{proposition}

\noindent Proposition \ref{5838} implies that a L\'evy process such that $d_H=0$ never creeps through a 
continuous non decreasing function. Note that since $f$ is non decreasing, when the process $X$ creeps 
through $f$, it necessarily creeps at its supremum, that is 
$\{X_{T_f}=f({T_f})\}=\{X_{T_f}=\overline{X}_{T_f}=f({T_f})\}$. 
 
\subsection{Creeping of L\'evy processes conditioned to stay positive}\label{2663}
Recall that $X$ is a non killed real L\'evy process such that $X_0=0$, a.s. Moreover, we assume again that $|X|$ 
is not a L\'evy process. The process $X$ conditioned to stay positive is a Doob $h$-transform of the process 
killed at its first hitting time of the negative half-line. Let us briefly recall the definition of this process of 
which one will find a more complete description in \cite{cd}. Let us set,
\[\tau_{-x}^-=\inf\{t>0:X_t<-x\}\,,\]
for $x\ge0$ and let $h(x)=\int_0^\infty\p(H_t^*\in[0,x])\,dt$ be the renewal function of the downward ladder 
height process $H^*$ of $X$ (see Subsection \ref{5399} for a full definition). Then the process $X$ conditioned 
to stay positive is a (possibly non conservative) strong Markov process with state space $(0,\infty)$, which we 
denote by $X^{x,\uparrow}$ when it is issued from $x$ and whose semigroup is given by 
\[\p(X^{x,\uparrow}_{t}\in dy)=\frac{1}{h(x)}\p(x+X_t\in dy,t<\tau_{-x}^-)\,,\;\;\;x,y>0\,.\]
It is proved in \cite{cd} that when 0 is regular for $(0,\infty)$ and $X$, the process $X^{x,\uparrow}$ 
converges weakly in the Skohorod's space, as $x$ tends to 0 toward a non degenerate process which will be 
denoted here by $X^{\uparrow}$. This process satisfies $X^{\uparrow}_0=0$, a.s.~and if moreover 
$\limsup_{t\rightarrow\infty}X_t=\infty$, a.s., then $(X^{x,\uparrow},x>0)$ is conservative and
$\lim_{t\rightarrow\infty}X^{\uparrow}_t=\infty$, a.s. Let us now define the future infimum process of
$X^\uparrow$ and its last passage time below the graph of a function $f$ respectively by,  
\[\underline{\underline{X}}_t^{\uparrow}=\inf_{s\ge t}X_s^{\uparrow}\,,\;\;t\ge0\;\;\;\mbox{and}\;\;\;
\sigma_f=\sup\{t:X^{\uparrow}_t\le f(t)\}\,.\]

\begin{theorem}\label{3737}
Let $X$ be a real L\'evy process such that $\limsup_{t\rightarrow\infty}X_t=\infty$, a.s. and assume that 
$0$ is regular for $(0,\infty)$. Let $X^\uparrow$ be the process $X$ conditioned to stay positive and 
let $f:(0,\infty)\rightarrow(0,\infty)$ be a continuous, non increasing function.
\begin{itemize}
\item[$1.$] Then $0<\sigma_f<\infty$, a.s.~and at time $\sigma_f$, the process $X^\uparrow$ can neither 
jump on the graph $\{(u,f(u)):u>0\}$ nor jump from this graph, that is 
\[\p(X_{\sigma_f-}^\uparrow=f(\sigma_f)<X_{\sigma_f}^\uparrow)=
\p(X_{\sigma_f-}^\uparrow<X_{\sigma_f}^\uparrow=f(\sigma_f))=0\,.\]
\item[$2.$] Assume that the renewal measure $U(dt,dx)$ of the ladder process $(\tau,H)$ has a continuous 
density $(t,x)\mapsto q_t^*(x)$ on $(0,\infty)^2$. Then for all $t_0$ and $t_1$ such that 
$0 \le t_0<t_1\le \infty$, 
\[\p(X_{\sigma_f}^\uparrow=\underline{\underline{X}}_{\sigma_f}^\uparrow=f({\sigma_f}),\,
t_0<{\sigma_f}<t_1)=d_H\int_{t_0}^{t_1}q^*_u(f(u)) \, du-d_\tau\int_{t_0}^{t_1}q^*_u(f(u))\, 
df(u)\,.\]
\end{itemize}
\end{theorem}

\noindent It appears that the creeping probability of $X^\uparrow$ at its future infimum through the 
graph of $f$ has the same expression as the creeping probability of $X$ at its past supremum stated 
in Theorem \ref{3636}. This can be explained by the following time reversal property 
\[[(X_{T_f}-X_{(T_f-t)-}\,,0\le t\le T_f)\,|\,X_{T_f}=\overline{X}_{T_f}]\,\ed\,[(X_t^\uparrow\,,0\le 
t\le \sigma_f)\,|\,X_{\sigma_f}^\uparrow=\underline{\underline{X}}_{\sigma_f}^\uparrow]\,,\]
where $X_{0-}=X_0=0$. The later identity in law is itself a consequence of the proof of Theorem \ref{3737}, 
see Section \ref{proof3}. 

\section{Examples}\label{examples}

Illustrating Theorem \ref{9049} with a bivariate subordinator $(Y,Z)$  such that none of the 
coordinates $Y$ and $Z$ is a pure drift and whose renewal density measure is explicit seems to be quite challenging.
Our first example regards Corollary \ref{2526}, that is the case when one of the subordinators is a pure drift.\\

a) Let $X$ be a stable subordinator with index $1/2$. Then for $t>0$, $X_t$ has density 
$$p_t(x)=\frac{t}{\sqrt{2\pi x^3}}  e^{\frac{-t^2}{2x}}\mathbbm{1}_{(0,\infty)}(x)\,.$$
From Corollary \ref{2526}, the distribution of the creeping time is then
\begin{eqnarray*}
\p(X_{T_f}=f({T_f}),\,T_f\in dt)=
-\frac{t}{\sqrt{2\pi f(t)^3}}e^{-t^2/(2f(t))}\,df(t)\,,\;\;\;t>0\,.
\end{eqnarray*}
Taking $f(t)=1/t^2$ gives the creeping probability through $f$ over the half-line $(0,\infty)$ and the 
distribution of the creeping time conditionally on this event:
\begin{eqnarray*}
\p(X_{T_f}=f({T_f}))&=&\sqrt{\frac2\pi}\int_0^\infty te^{-t^4/2}\,dt=\frac12
\\
\p(T_f\in dt\,|\,X_{T_f}=f({T_f}))&=&2\sqrt{\frac2\pi}te^{-t^4/2}\,dt\,,\;\;\;t>0\,.\\
\end{eqnarray*}
Our second example illustrates Theorem \ref{3636}.\\

b) Standard Brownian motion creeping at its supremum. When $X$ is standard Brownian motion, $0$ is regular for 
both half-lines $(-\infty,0)$ and $(0,\infty)$ and hence $d_\tau=0$. On the other hand, our normalisation of the 
local time $L$ (see Subsection \ref{5399}) gives $d_H=1/\sqrt{2}$ and for $t>0$,
\[q_t^*(x)=\frac x{\sqrt{\pi t^3}}e^{-x^2/2t}\mathbbm{1}_{(0,\infty)}(x)\,.\]
From part 2.~of Theorem \ref{3636}, the distribution of the creeping time at the supremum with respect to 
any continuous non increasing function $f:(0,\infty)\rightarrow(0,\infty)$ is given by 
\begin{eqnarray*}
\p(X_{T_f}=\overline{X}_{T_f}=f({T_f}),\,T_f\in dt)=
\frac {f(t)}{\sqrt{2\pi t^3}}e^{-f(t)^2/2t}\,dt\,,\;\;\;t>0\,.
\end{eqnarray*}
For instance, the probability for the standard Brownian motion to creep at its supremum through the function 
$f(t)=1/t$ on $(0,\infty)$ and the distribution of the creeping time $T_f$ conditionally on this event are 
respectively, 
\begin{eqnarray*}
\p(X_{T_f}=\overline{X}_{T_f}=f({T_f}))&=&\int_0^\infty\frac {1}{\sqrt{2\pi t^5}}e^{-1/2t^3}\,dt=\frac13\\
\p(T_f\in dt\,|\,X_{T_f}=\overline{X}_{T_f}=f({T_f}))&=&\frac{3}{\sqrt{2\pi}t^{5/2}}e^{-1/2t^3}\,dt\,,\;\;\;t>0\,.
\end{eqnarray*}
A slight generalisation of the above computation can be obtained by considering the case of Brownian motion 
with drift $\mu\in\mathbb{R}$. The expression of $q_t^*(x)$ is then $\frac x{\sqrt{\pi t^3}}e^{-(x+\mu)^2/2t}$.
The density $(t,x)\mapsto q_t^*(x)$ of the renewal measure can actually be made explicit in very few cases and 
the few known expressions are too complicated to allow the calculation of the creeping probability.\\

Actually both examples a) and b) are closely related to each other. Indeed, for standard Brownian motion,
the ladder process $(\tau,H)$ is such that $\tau$ is a stable subordinator with index $1/2$ and $H$ is a pure
drift process, that is $H_t=d_H\cdot t$, $t\ge0$. So, example b) is also a consequence of Corollary \ref{2526}.

\section{Application : creeping of $\alpha$-stable Ornstein-Uhlenbeck processes}\label{OU}

Let $X$ be a non killed real L\'evy process such that $|X|$ is not a subordinator and $X_0=0$, a.s. 
The Ornstein-Uhlenbeck process starting from $z\in\mathbb{R}$ and driven by $X$ is the unique strong solution 
of the following stochastic differential equation:
\begin{equation}\label{2227}
 Z_t=z+X_t-\gamma \int_0^t Z_s ds\,,\;\;t\ge0\,,\;\;\gamma>0\,.
\end{equation}
It is explicitly given by $Z_t=e^{-\gamma t}(z+\int_0^t e^{\gamma s} dX_s)$. For $x\in\mathbb{R}$, we set 
$T^+_x=\inf\{t:Z_t>x\}$ and $T^-_x=\inf\{t:Z_t<x\}$ and we say that $Z$ creeps 
through the state $x\neq z$ if either $x>z$ and $\p(Z_{T_x^+}=x,\,T_x^+<\infty)>0$ or $x<z$ and
$\p(Z_{T_x^-}=x,\,T_x^-<\infty)>0$. We are interested 
here in the creeping property of $Z$, when $X$ is a stable L\'evy process with index $\alpha\in(0,1)$.
Actually stable L\'evy processes never creep across any level. Creeping of an Ornstein-Uhlenbeck 
process $Z$ driven by a stable L\'evy process with index $\alpha\in (0,1)$ is actually due to the smoothing 
effect of the functional $t\mapsto\int_0^t Z_s ds$, as the proof of the next proposition shows. Recall 
that when $X$ is stable with any index $\alpha\in(0,2]$, $Z$ fulfils the following representation: 
\begin{equation}\label{2562}
(Z_t,\,t\ge0)\ed\left(e^{-\gamma t}(z+X_{\frac{e^{\alpha\gamma t}-1}{\alpha\gamma}}),\,t\ge0\right)\,,
\end{equation}
see for instance E 18.17, p.116 in \cite{sa}. When moreover $X$ is symmetric and $\alpha\in(0,1)$, it is proved 
in \cite{ja} (see Proposition 12 on p.618) that $Z$ creeps through $x\in(z,0)$ when $z<0$ and through $x\in(0,z)$,
when $z>0$. Proposition \ref{5446} below provides an extension of this result. 

\begin{proposition}\label{5446} Let $Z$ be the Ornstein-Uhlenbeck process solution of $(\ref{2227})$ with 
$\gamma>0$, starting point $z\in\mathbb{R}$ and driven by a stable L\'evy process $X$ with index $\alpha\in(0,1)$
such that $|X|$ is not a subordinator. 
\begin{itemize}
\item[$1.$] If $z \in (-\infty,0)$, then $Z$ creeps through $x$ if and only if $x\in(z,0)$.
\item[$2.$] If $z \in (0,\infty)$, then $Z$ creeps through $x$ if and only if $x\in(0,z)$.
\item[$3.$] If $z=0$, then $Z$ does not creep through any level.
\end{itemize}
\end{proposition}
\begin{proof} 
From the representation $(\ref{2562})$, we can write by setting $s=(e^{\alpha\gamma t}-1)/(\alpha\gamma)$,
\begin{eqnarray*}
T^+_x&=&\frac1{\alpha\gamma}\log(\alpha\gamma\inf\{s:X_s>x(\alpha\gamma s+1)^{1/\alpha}-z\}+1)\\
T^-_x&=&\frac1{\alpha\gamma}\log(\alpha\gamma\inf\{s:X_s<x(\alpha\gamma s+1)^{1/\alpha}-z\}+1)\,,
\end{eqnarray*}
so that $Z$ creeps through $x$ if and only if $X$ creeps through $f(s)=x(\alpha\gamma s+1)^{1/\alpha}-z$.

Let $z \in (-\infty,0)$ and $x\in(z,0)$. Then recall that $X$ has bounded variation and zero drift and that
for all $t>0$, $\p(X_t\in dx)$ is absolutely continuous on $\mathbb{R}$ with a bounded density, see Theorem 
53.1, p.404 in \cite{sa}. Therefore, $X$ satisfies the conditions of Corollary \ref{9445}. On the other hand, 
note that $f(0)=x-z>0$ and that $f$ is continuous. Moreover, $f(s)+as$ is non increasing on $(0,\infty)$
whenever $0<a<-x\gamma$. Let us take such a value $a$ and note that there is $t_1>0$ such that $f(s)+as$ is 
positive on $(0,t_1)$. Then the result follows from Corollary \ref{9445}.  

If $z<0\le x$, then $f(0)>0$ and $f$ is non decreasing. Since $X$ satisfies $d_H=0$ ($H$ is a stable subordinator), 
the result follows from Proposition \ref{5838}.

If $x<z<0$ then $f(0)<0$ and $f$ is non increasing and the result follows again from Proposition \ref{5838} 
by considering $-X$ and $-f$. 

The proof of $2.$ is deduced from this of $1.$ by considering the Ornstein-Uhlenbeck process $-Z$ which is 
solution of $(\ref{2227})$ with $\gamma>0$, starting point $-z\in\mathbb{R}\setminus\{0\}$ and driven by the 
stable process $-X$. 

If $z=0$ and $x>0$, then $f$ is non decreasing and the result follows again from Proposition \ref{5838}. 
For $x<0$, we derive the result by considering $-Z$ from the same argument as above. 
\end{proof} 
\noindent Looking at the previous proof and in accordance with the remarks made in Subsection \ref{6525}, we 
conjecture that when $\alpha>1$, since $X$ has unbounded variation, the process $Z$ never creeps across any 
level. However, as proved in \cite{cu}, points are not polar for $Z$ when the leading L\'evy process is stable 
with any index $\alpha\in(0,2]$.\\ 

Finally let us mention that dealing with the creeping through $x\in(z,0)$ of an Ornstein-Uhlenbeck process $Z$ 
starting from $z<0$ and driven by any L\'evy process $X$ would require to study the creeping property of
$X$ through the continuous non increasing adapted functional $t\mapsto x-z+\gamma\int_0^t Z_s ds$, as equation 
(\ref{2227}) shows. This raises the more general question of the creeping of a L\'evy process through any adapted, 
continuous and non increasing functional.

\section{Proof of Theorem \ref{9049} and Proposition \ref{8983}}\label{proof}

Let us note that from monotone convergence, we do not lose any generality by proving Theorem \ref{9049} 
in the case where $t_0,u_0>0$ and $t_1,u_1<\infty$. Therefore in all this section, $t_0$, 
$u_0$, $t_1$ and $u_1$ will be chosen so that
\[0<u_0<u_1<\infty\;\;\;\mbox{and}\;\;\;0<t_0<t_1<\infty\,.\] 

\subsection{Proof of the first assertion in Theorem \ref{9049}}\label{9790}
The following lemma actually shows a much stronger result than the first assertion of Theorem \ref{9049}. 
It will also be used in further results. 

\begin{lemma}\label{5698}
Let ${\mathcal L}$ be any $\mathbb{R}^d$-valued L\'evy process with infinite lifetime and set
$\Delta {\mathcal L}_t={\mathcal L}_t-{\mathcal L}_{t-}$, $t>0$. Let $G:\mathbb{R}^d\rightarrow\mathbb{R}$ be any 
Borel function such that for all $x\in\mathbb{R}^d$, $\p(G({\mathcal L}_t+x)=0)=0$,  for $\lambda$-a.e. $t\ge0$ 
$($$\lambda$ being the Lebesgue measure on $[0, \infty)$$)$. Then
\begin{equation}\label{9545}
\p\left(\exists\; t\ge0,\;G({\mathcal L}_{t-})=0,\,\Delta {\mathcal L}_t\neq0\right)=\p(\exists\; 
t\ge0,\;G({\mathcal L}_{t})=0,\,\Delta {\mathcal L}_t\neq0)=0\,.
\end{equation}
\end{lemma} 
\begin{proof} Recall that $t\mapsto \Delta {\mathcal L}_t$ is a 
Poisson point process with characteristic measure $\nu$, where $\nu$ is the L\'evy measure of ${\mathcal L}$. 
Then from the compensation formula for Poisson point processes, for every $\varepsilon>0$, 
\[\e\left(\sum_{t\ge0}\ind_{\{G({\mathcal L}_{t-})=0\}}\ind_{\{\|\Delta {\mathcal L}_t\|>\varepsilon\}}\right)=
\nu(\{x:\|x\|>\varepsilon\})\e\left(\int_0^\infty\ind_{\{G({\mathcal L}_{t})=0\}}\,dt\right).\]
But from our assumption, the last term is 0. 
Then the equality 
\[\p\left(\exists\; t\ge0,\;G({\mathcal L}_{t-})=0,\,\Delta {\mathcal L}_t\neq0\right)=0\] 
in (\ref{9545}) is obtained by taking $\varepsilon$ to 0 and using monotone convergence.\\

We prove the second equality through the same argument by writing
\begin{eqnarray*}
\e\left(\sum_{t\ge0}\ind_{\{G({\mathcal L}_{t-}+\Delta {\mathcal L}_t)=0\}}\ind_{\{\|\Delta 
{\mathcal L}_t\|>\varepsilon\}}\right)
&=&\int_{0}^\infty dt\e\left(\int_{\{\|x\|>\varepsilon\}}\nu(dx)\ind_{\{G({\mathcal L}_{t-}+x)=0\}}\right)\\
&=&\int_{\{\|x\|>\varepsilon\}} \nu(dx)\int_{0}^\infty\p(G({\mathcal L}_{t}+x)=0)\,dt=0
\end{eqnarray*}
and the conclusion follows as for the first equality. 
\end{proof}

Let us now prove the first assertion of Theorem \ref{9049}. First of all, since $f$ is positive and non increasing, 
it is clear that $0<S<\infty$, a.s., so that $Y_{S-}$, $Z_{S-}$ are well defined and $Y_{S-}<\infty$ and 
$Z_{S-}<\infty$, a.s.  Then let us apply Lemma \ref{5698} to the L\'evy process $\mathcal{L}=(Y,Z)$ and the Borel 
function $G(y,z)=f(y)-z$ under the assumption that $\mathcal{L}$ has infinite lifetime. The condition: for all 
$x=(y,z)\in\mathbb{R}^2$, $\p(G(\mathcal{L}_t+x)=0)=\p(f(Y_t+y)=Z_t+z)=0$, 
for $\lambda$-a.e. $t\ge0$ is satisfied since, by assumption, the renewal measure $v(dy,dz)$ of $(Y,Z)$ 
is absolutely continuous. Then part 1.~of Theorem \ref{9049} follows from Lemma \ref{5698} and the equalities 
\begin{eqnarray*}
\p(Z_{S-}=f(Y_{S-}),\Delta(Y,Z)_{S}\neq0)&=&\p(\exists t \ge0,\,Z_{t-}=f(Y_{t-}),\Delta(Y,Z)_t \neq0)\\
\p(Z_{S}=f(Y_{S}),\Delta(Y,Z)_{S}\neq0)&=&\p(\exists t \ge0,\,Z_{t}=f(Y_{t}),\Delta(Y,Z)_t \neq0)\,.
\end{eqnarray*}
Now let $\tilde{\mathcal{L}}=(\tilde{Y},\tilde{Z})$ be a L\'evy process with infinite lifetime such that $\mathcal{L}$ 
is obtained by killing $\tilde{\mathcal{L}}$ at an independent exponential time ${\bf e}$. Denote by $\zeta$ 
the lifetime of ${\mathcal L}$. Then,
\begin{eqnarray*}
\p(Z_{S-}=f(Y_{S-}),\Delta(Y,Z)_{S}\neq0)&=&\p(\exists t \le \zeta,\,Z_{t-}=f(Y_{t-}),\Delta\mathcal{L}_t \neq0)\\
&=&\p(\exists t \le {\bf e},\,\tilde{Z}_{t-}=f(\tilde{Y}_{t-}),\Delta\tilde{\mathcal{L}}_t \neq0)\\
&\le&\p(\tilde{Z}_{S-}=f(\tilde{Y}_{S-}),\Delta(\tilde{Y},\tilde{Z})_{S}\neq0)=0\,.
\end{eqnarray*}
The case of $\p(Z_{S}=f(Y_{S}),\Delta(Y,Z)_{S}\neq0)=0$ is handled in the same way.

\subsection{Preliminary lemmas}\label{7655}
Recall that $(Y,Z)$ is a bivariate subordinator as defined in Subsection \ref{3025}. Its semigroup and its renewal 
measure are respectively denoted by $p_t(dy,dz)$ and $v(dy,dz)$. In all the remainder of this paper we set for all 
$y,z\ge0$,
\[T_y^Y=\inf\{t:Y_t>y\}\;\;\;\mbox{and}\;\;\;T_z^Z=\inf\{t:Z_t>z\}\,.\] 

\begin{lemma}\label{8333} Assume that the renewal measure $v(dy,dz)$ is absolutely continuous on $(0,\infty)^2$ 
and let us denote by $v(y,z)$ its density, then for $y,z\in(0,\infty)$,
\begin{eqnarray}
&&\p(Y_{T_y^Y}=y,\,Z_{T_y^Y}\in dz)\,dy=d_Yv(y,z)\,dy\,dz\label{2621}\\
&&\p(Z_{T_z^Z}=z,\,Y_{T_z^Z}\in dy)\,dz=d_Zv(y,z)\,dy\,dz\,.\label{8945}
\end{eqnarray}
Assume that for all $t>0$, the semigroup $p_t(dy,dz)$ is absolutely continuous on $(0,\infty)^2$ and let us denote 
by $p_t(y,z)$ its densities, then for $t,y,z\in(0,\infty)$,
\begin{eqnarray}
&&\p(Y_{T_y^Y}=y,\,Z_{T_y^Y}\in dz,\,T_y^Y\in dt)\,dy=d_Yp_t(y,z)\,dt\,dy\,dz\label{2622}\\
&&\p(Z_{T_z^Z}=z,\,Y_{T_z^Z}\in dy,\,T_z^Z\in dt)\,dz=d_Zp_t(y,z)\,dt\,dy\,dz\,.\label{8946}
\end{eqnarray}
\end{lemma}
\begin{proof} Let us prove (\ref{2621}) and (\ref{2622}). Note that the time $T_y^Y$ satisfies 
$d_YT_y^Y=\int_0^y\ind_{\{Y_{T_s^Y}=s\}}\,ds$, for all $y\ge0$.  Hence for $\alpha,\beta,\gamma>0$, 
\begin{eqnarray*}
&&\int_{y=0}^\infty\int_{t,z\in[0,\infty)^2} e^{-\alpha y-\beta z-\gamma t}
\p(Y_{T_y^Y}=y,\,Z_{T_y^Y}\in dz,\,T_y^Y\in dt)\,dy\\
&=&\e\left(\int_0^\infty e^{-\alpha Y_{T_y^Y}-\beta Z_{T^Y_y}-\gamma T_y^Y}d_YdT_y^Y\right)=
\e\left(\int_0^\infty e^{-\alpha Y_t-\beta Z_{t}-\gamma t}d_Ydt\right)\\
&=&\int_{y=0}^\infty\int_{t,z\in[0,\infty)^2} e^{-\alpha y-\beta z-\gamma t}d_Yp_t(dy,dz)\,dt\,,
\end{eqnarray*}
which proves both identities. The proof of (\ref{8945}) and (\ref{8946}) is the same.  
\end{proof}
Let us mention that expressions similar to (\ref{2621}) and (\ref{8945}) have been obtained in 
Section 4 of \cite{gm}, see Theorem 4.2 therein. We also use the latter theorem for the proof of the next Lemma.

\begin{lemma}\label{2833} 
Assume that the renewal measure $v(dy,dz)$ is absolutely continuous on $(0,\infty)^2$ and that its density 
$(y,z)\mapsto v(y,z)$ is locally bounded. Then there are functions $\varepsilon_1$ and $\varepsilon_2$ 
satisfying $\varepsilon_1(h),\varepsilon_2(h)\rightarrow0$, as $h\rightarrow0$ and such that for all 
$h\in(0,u_1-u_0)$ and $u\in[u_0,u_1-h]$, 
\begin{eqnarray*}
&&\p(Y_{T_{u+h}^Y-}\in(u,u+h),\,Z_{T_{u+h}^Y-}\in (f(u+h),f(u)))\le h\varepsilon_1(h)\,,\\\label{2971}
&&\p(Z_{T_{f(u)}^Z-}\in(f(u+h),f(u)),\,Y_{T_{f(u)}^Z-}\in (u,u+h))\le h\varepsilon_2(h)\,.\label{8836}
\end{eqnarray*}
\end{lemma}
\begin{proof}
Let us denote by $\pi_Y$ and $\pi_Z$ the L\'evy measures of $Y$ and $Z$ and set $\bar{\pi}_Y(y)=\pi_Y(y,\infty)$ 
and $\bar{\pi}_Z(z)=\pi(z,\infty)$. Recall that $\zeta$ is the lifetime of $(Y,Z)$ and let us denote by 
$q\in[0,\infty)$ its rate. Applying Theorem 4.2 in \cite{gm} we obtain,
\begin{eqnarray}
&&\p(Y_{T_{u+h}^Y-}\in(u,u+h),\,Z_{T_{u+h}^Y-}\in (f(u+h),f(u)),\,T_{u+h}^Y<\zeta)\nonumber\\
&&=\int_{z=f(u+h)}^{f(u)}\int_{y=u}^{u+h}v(y,z)\bar{\pi}_Y(y)\,dydz
\le\bar{\pi}_Y(u_0)Ch\sup_{u\in[u_0,u_1-h]}(f(u)-f(u+h))\,,\label{8466}
\end{eqnarray}
where $C$ is a bound for $v$ on $[u_0,u_1]\times[f(u_1),f(u_0)]$ (recall that $u_0>0$). 
On the other hand,
\begin{eqnarray}
&&\p(Y_{T_{u+h}^Y-}\in(u,u+h),\,Z_{T_{u+h}^Y-}\in (f(u+h),f(u)),\,T_{u+h}^Y=\zeta)\nonumber\\
&&\quad\le\p(Y_{\zeta-}\in(u,u+h),\,Z_{\zeta-}\in (f(u+h),f(u)))\nonumber\\
&&\quad=q\int_{z=f(u+h)}^{f(u)}\int_{y=u}^{u+h}v(y,z)\,dydz\le qCh\sup_{u\in[u_0,u_1-h]}(f(u)-f(u+h))\,.\label{8467}
\end{eqnarray}
Then the first inequality of the statement follows from (\ref{8466}), (\ref{8467}) and the uniform continuity 
of $f$ on $[u_0,u_1]$. 

Similarly, we have from Theorem 4.2 in \cite{gm},
\begin{eqnarray*}
&&\p(Z_{T_{f(u)}^Z-}\in(f(u+h),f(u)),\,Y_{T_{f(u)}^Z-}\in (u,u+h),\,T_{f(u)}^Z<\zeta)\\
&&=\int_{z=f(u+h)}^{f(u)}\int_{y=u}^{u+h}v(y,z)\bar{\pi}_Z(z)\,dydz\le\bar{\pi}_Z(f(u_1))
Ch\sup_{u\in[u_0,u_1-h]}(f(u)-f(u+h))\,,
\end{eqnarray*}
and the same conclusion follows from the same argument.
\end{proof}

\subsection{Proof of the lower bound in identities (\ref{7229}) and (\ref{6366})}\label{8590}
Let us first prove the lower bound in equation (\ref{6366}), that is the inequality, 
\begin{eqnarray}
&&\p(Z_{S}=f(Y_S),\,u_0<Y_{S}<u_1,\,t_0<S<t_1)\geqslant\nonumber\\
&&\qquad\qquad d_Z\int_{u_0}^{u_1}\int_{t_0}^{t_1}p_t(u,f(u))\,dt \, du-d_Y\int_{u_0}^{u_1}\label{6311}
\int_{t_0}^{t_1}p_t(u,f(u))\,dt\, df(u)\,.
\end{eqnarray}
We define the sequence $(\sigma_n)_{n\ge1}$ of subdivisions of $[u_0,u_1]$, 
\begin{equation}\label{6392}
\sigma_n=\left\lbrace u_k^{n}=u_0 +\frac{k}{n} (u_1-u_0), \ 0 \leqslant k \leqslant n \right\rbrace,
\;\;\;n\ge1\,,
\end{equation}
and we set, for all $n\ge1$, $\displaystyle B_n=\bigcup_{k=0}^{n-1} B_{n,k}^{(1)} \cup B_{n,k}^{(2)}$, 
where
\begin{eqnarray}
B_{n,k}^{(1)}&=&\left\lbrace Z_{T_{f(u_k^n)}^Z}=f(u_k^n), u_k^n \leqslant Y_{T_{f(u_k^n)}^Z}<u_{k+1}^n,\, 
t_0<T_{f(u_k^n)}^Z<t_1\right\rbrace\label{3722}\\
B_{n,k}^{(2)}&=&\left\lbrace Y_{T_{u_{k+1}^n}^Y}=u_{k+1}^n, f(u_{k+1}^n)\leqslant Z_{T_{u_{k+1}^n}^Y}<
f(u_k^n),\, t_0<T_{u_{k+1}^n}^Y<t_1\right\rbrace\,.\label{9226}
\end{eqnarray}
Then let us check that 
\begin{equation}\label{7366}
\limsup_n B_n \subset \lbrace Z_{S}=f(Y_S),\,u_0\le Y_{S}\le u_1,\,t_0\le S\le t_1\rbrace\,, 
\;\;\;\mbox{almost surely}\,.
\end{equation}
Let us first note that 
\begin{eqnarray*}
&&\omega\in B_{n,k}^{(1)}\Rightarrow\|(Y_{T_{f(u_k^n)}^Z},Z_{T_{f(u_k^n)}^Z})-(u_k^n,f(u_k^n))\|\le 
u_{k+1}^n-u_{k}^n=\frac{u_1-u_0}n\;\;\;\mbox{and}\\
&&\omega\in B_{n,k}^{(2)}\Rightarrow\|(Y_{T_{u_{k+1}^n}^Y},Z_{T_{u_{k+1}^n}^Y})-(u_{k+1}^n,
f(u_{k+1}^n))\|\le f(u_k^n)-f(u_{k+1}^n)\le\varepsilon(n)\,,
\end{eqnarray*}
where in the last inequality, the mapping $n\mapsto\varepsilon(n)$ does not depend on $k$ and satisfies
$\lim_{n\rightarrow\infty}\varepsilon(n)=0$. This follows from the uniform continuity of $f$ on $[u_0,u_1]$. 
Therefore if $\omega\in\limsup_n B_n$, then there is a subsequence $u_{j_n}^{i_n}(\omega):=u_n$ of 
$(u_k^n)_{n\ge1,\,0\le k\le n}$ which converges to $u\in[u_0,u_1]$ and such that 
$(Y_{T_{f(u_n)}^Z},Z_{T_{f(u_n)}^Z})$ converges to $(u,f(u))$ or $(Y_{T_{u_n}^Y},Z_{T_{u_n}^Y})$ 
converges to $(u,f(u))$. Assume for instance that $(Y_{T_{f(u_n)}^Z},Z_{T_{f(u_n)}^Z})$ converges to
$(u,f(u))$. Since $Y$ and $Z$ are non decreasing, this means that $T_{f(u_n)}^Z$ tends to some value $S$ 
such that $(Y_{S-},Z_{S-})=(u,f(u))$. Moreover $t_0\le S\le t_1$ and $u_0\le Y_{S-}\le u_1$. In particular, 
$Z_{S-}=f(Y_{S-})$, so that $S=\inf\{t:Z_t\ge f(Y_t)\}$, since $Y$ and $Z$ are non decreasing. 
If $(Y_{T_{u_n}^Y},Z_{T_{u_n}^Y})$ converges to $(u,f(u))$, then the same argument leads to the same 
conclusion. Then (\ref{7366}) follows from part 1.~of Theorem \ref{9049}.

Now, let us prove that
\begin{equation}\label{lim_1}
\lim_{n\rightarrow+\infty}\p(B_n)=
d_Z\int_{u_0}^{u_1}\int_{t_0}^{t_1}p_t(u,f(u))\,dt \, du-d_Y\int_{u_0}^{u_1}
\int_{t_0}^{t_1}p_t(u,f(u))\,dt\, df(u)\,.
\end{equation}
\noindent We first notice that, for all $n \in \N$ and all integer 
$0\leqslant k \leqslant n-1$, $B_{n,k}^{(1)} \cap B_{n,k}^{(2)}= \emptyset$. Moreover, $\sigma_n$ being a 
subdivision of $[u_0,u_1]$, for all integers $n\ge2$, $k'$, $k''$ such that $k'\neq k''$ and 
$0\leqslant k',k''\leqslant n-1$,  
$(B_{n,k'}^{(1)}\cup B_{n,k'}^{(2)})\cap(B_{n,k''}^{(1)}\cup B_{n,k''}^{(2)})=\emptyset$. On the other hand, 
from $(\ref{2622})$ and $(\ref{8946})$ in Lemma \ref{8333}, we obtain  
\[\p \left( B_{n,k}^{(1)} \right)=d_Z\int_{u_k^n}^{u_{k+1}^n}\int_{t_0}^{t_1}p_t(u,f(u_k^n))\,dt\,du\;\;\;
\mbox{and}\;\;\;
\p\left( B_{n,k}^{(2)} \right)=d_Y\int_{f(u_{k+1}^n)}^{f(u_k^n)}\int_{t_0}^{t_1}p_t(u_{k+1}^n,s)\,dt\,ds\,,\]
so that
\begin{eqnarray*}
\p(B_n)&=&\sum_{k=0}^{n-1} \left[\p\left( B_{n,k}^{(1)} \right) + \p \left( B_{n,k}^{(2)} \right) \right]\\
&=&d_Z\sum_{k=0}^{n-1}\int_{u_k^n}^{u_{k+1}^n}\int_{t_0}^{t_1}p_t(u,f(u_k^n))\,dt\,du+
d_Y\sum_{k=0}^{n-1}\int_{f(u_{k+1}^n)}^{f(u_k^n)}\int_{t_0}^{t_1}p_t(u_{k+1}^n,s)\,dt\,ds\,.
\end{eqnarray*}
Recall that $0<t_0<t_1<\infty$ and $0<u_0<u_1<\infty$ and assume first that $f(u_1)<f(u_0)$, 
in particular $0<f(u_1)<f(u_0)<\infty$. 
Then we derive from the continuity of $(t,y,z)\mapsto p_t(y,z)$ on 
$[t_0,t_1]\times[u_0,u_1]\times[f(u_1),f(u_0)]$ that the mapping $(y,z)\mapsto \int_{t_0}^{t_1}p_t(y,z)\,dt$  
is uniformly continuous on $[u_0,u_1]\times[f(u_1),f(u_0)]$. Set $\bar{v}(y,z):=\int_{t_0}^{t_1}p_t(y,z)\,dt$. 
We derive from the uniform continuity of $(y,z)\mapsto\bar{v}(y,z)$ on $[u_0,u_1]\times[f(u_1),f(u_0)]$ and 
this of the mapping $u\mapsto f(u)$ on $[u_0,u_1]$ that there are functions $\alpha(n),\,\beta(n)\rightarrow0$, 
as $n\rightarrow\infty$, such that for all $0\le k\le n-1$, 
$|\bar{v}(u,f(u_k^n))-\bar{v}(u_k^n,f(u_k^n))|\le\alpha(n)$, 
for all $u\in[u_k^n,u_{k+1}^n]$ and $|\bar{v}(u_{k+1}^n,s)-\bar{v}(u_{k+1}^n,f(u_{k+1}^n))|\le\beta(n)$, for 
all $s\in[f(u_{k+1}^n),f(u_k^n)]$. These arguments and the Riemann-Stieltjes integrability of the function 
$u\mapsto\bar{v}(u,f(u))$ with respect to the measures $du$ and $df(u)$ imply that 
\[\lim_{n\rightarrow\infty}\sum_{k=0}^{n-1}\int_{u_k^n}^{u_{k+1}^n} \bar{v}(u,f(u_k^n))\,du=
\lim_{n\rightarrow\infty}\sum_{k=0}^{n-1}\bar{v}(u_k^n,f(u_k^n))(u_{k+1}^n-u_k^n)=\int_{u_0}^{u_1}\bar{v}(u,f(u)) 
\,du\,,\]
and that
\begin{eqnarray}
\lim_{n\rightarrow\infty}\sum_{k=0}^{n-1}\int_{f(u_{k+1}^n)}^{f(u_k^n)} \bar{v}(u_{k+1}^n,s)\,ds&=&
\lim_{n\rightarrow\infty}\sum_{k=0}^{n-1}\bar{v}(u_{k+1}^n,f(u_{k+1}^n))(f(u_k^n)-f(u_{k+1}^n))\nonumber\\
&=&-\int_{u_0}^{u_1}\bar{v}(u,f(u))\,df(u)\,,\label{2689}
\end{eqnarray}
which proves (\ref{lim_1}) in this case. If $f(u_0)=f(u_1):=w$, then the result follows from the same arguments.
We only need to replace the mapping $(y,z)\mapsto\int_{t_0}^{t_1}p_t(y,z)\,dt$ by the mapping 
$y\mapsto\int_{t_0}^{t_1}p_t(y,w)\,dt$ and  the function $u\mapsto\bar{v}(u,f(u))$ by  the function 
$u\mapsto\bar{v}(u,w)$. (The term (\ref{2689}) is then clearly equal to 0.) Finally, recall that 
$\displaystyle \limsup_{n \longrightarrow \infty} B_n \subset 
\lbrace Z_{S}=f(Y_S),\,u_0\le Y_{S}\le u_1,\,t_0\le S\le t_1\rbrace$, a.s. Then by Fatou's lemma, 
\begin{equation*}
\mathbb{P}(Z_{S}=f(Y_S),\,u_0\le Y_{S}\le u_1,\,t_0\le S\le t_1)\geqslant 
\p(\limsup_{n \longrightarrow \infty} B_n)\geqslant \lim_{n \longrightarrow \infty} \p(B_n)\,, 
\end{equation*}
and (\ref{lim_1}) implies that
\begin{eqnarray*}
&&\p(Z_{S}=f(Y_S),\,u_0\le Y_{S}\le u_1,\,t_0\le S\le t_1)\geqslant\nonumber\\
&&\qquad\qquad d_Z\int_{u_0}^{u_1}\int_{t_0}^{t_1}p_t(u,f(u))\,dt \, du-d_Y\int_{u_0}^{u_1}
\int_{t_0}^{t_1}p_t(u,f(u))\,dt\, df(u)\,.
\end{eqnarray*}
Since the latter inequality is valid for all $u_0$, $u_1$, $t_0$ and $t_1$, large inequalities in the
left member can be replaced by strict inequalities, so that (\ref{6311}) holds.

The proof of the corresponding inequality in (\ref{7229}) uses $(\ref{2621})$ and $(\ref{8945})$ instead 
of $(\ref{2622})$ and $(\ref{8946})$ in the above arguments. It is actually simpler so it is omitted.

\subsection{Proof of the upper bound in identities (\ref{7229}) and (\ref{6366})}\label{9791}
Let us first prove the upper bound in equation (\ref{6366}), that is the inequality
\begin{eqnarray}
&&\p(Z_{S}=f(Y_S),\,u_0<Y_{S}<u_1,\,t_0<S<t_1)\leqslant\nonumber\\
&&\qquad\qquad d_Z\int_{u_0}^{u_1}\int_{t_0}^{t_1}p_t(u,f(u))\,dt \, du-d_Y\int_{u_0}^{u_1}\label{6293}
\int_{t_0}^{t_1}p_t(u,f(u))\,dt\, df(u)\,.
\end{eqnarray}
Recall the definition (\ref{6392}) of the sequence $(\sigma_n)_{n\ge1}$ 
of subdivisions of $[u_0,u_1]$. We define the two following sequences of events which are slightly 
different from those introduced in (\ref{3722}) and (\ref{9226}), 
\begin{eqnarray*}
\overline{B}_{n,k}^{(1)}&=&\left\lbrace Z_{T_{f(u_k^n)}^Z}=f(u_k^n), u_k^n \leqslant 
Y_{T_{f(u_k^n)}^Z}\leqslant u_{k+1}^n,\, t_0\leqslant T_{f(u_k^n)}^Z\leqslant t_1+o_1(n)\right\rbrace\\
\overline{B}_{n,k}^{(2)}&=&\left\lbrace Y_{T_{u_{k+1}^n}^Y}=u_{k+1}^n, f(u_{k+1}^n)\leqslant 
Z_{T_{u_{k+1}^n}^Y}\leqslant f(u_k^n),\, t_0\leqslant T_{u_{k+1}^n}^Y\leqslant t_1+o_2(n)\right\rbrace\,,
\end{eqnarray*}
where $o_1(n)=d_Z^{-1}\max_{0\le k\le n-1}(f(u_{k}^n)-f(u_{k+1}^n))$, if $d_Z>0$ and $o_2(n)=(nd_Y)^{-1}$, 
if $d_Y>0$. If $d_Y=0$ (resp.~$d_Z=0$), we set $o_2(n)=\infty$ (resp.~$o_1(n)=\infty$), for all $n\ge1$. 
Note that from uniform continuity of $f$ on $[u_0,u_1]$, if $d_Z>0$, then $o_1(n)\rightarrow0$, as 
$n\rightarrow\infty$ (which is also clearly verified for $o_2$, when $d_Y>0$). Then let us also define the 
sequences,
\begin{eqnarray*}
B_{n,k}^{(3)}&=&\left\lbrace f(u_{k+1}^n)\leqslant Z_{T_{f(u_k^n)}^Z-}\leqslant f(u_k^n),\,
u_k^n\leqslant Y_{T_{f(u_k^n)}^Z-}\leqslant u_{k+1}^n\right\rbrace\\
B_{n,k}^{(4)}&=&\left\lbrace u_k^n \leqslant Y_{T_{u_{k+1}^n}^Y-}\leqslant u_{k+1}^n,\,
f(u_{k+1}^n)\leqslant Z_{T_{u_{k+1}^n}^Y-}\leqslant f(u_k^n)\right\rbrace\,,
\end{eqnarray*}
for $n\ge1$ and $k=0,\dots,n-1$ and let us check the inclusion,
\begin{equation}\label{9690}
A:=\{Z_{S}=f(Y_S),\,u_k^n\leqslant Y_S< u_{k+1}^n,\,t_0<S<t_1\}\subset \overline{B}_{n,k}^{(1)}\cup 
\overline{B}_{n,k}^{(2)}\cup B_{n,k}^{(3)}\cup B_{n,k}^{(4)}\,.
\end{equation}
Assume that $A$ holds and note that $T_{f(u_k^n)}^Z\ge S$ and $T_{u_{k+1}^n}^Y\ge S$. Then, only the
two events described below in $(a)$ and $(b)$ can occur: 

$(a)$ In the first case, either $Z$ creeps through the level $f(u_k^n)$ before time $T_{u_{k+1}^n}^Y$ or $Y$ 
creeps through the level $u_{k+1}^n$ before time $T_{f(u_k^n)}^Z$. 

Let us consider the subcase where $Z$ creeps through the level $f(u_k^n)$ before time  $T_{u_{k+1}^n}^Y$. 
Then we have $Z_{T_{f(u_k^n)}^Z}=f(u_k^n)$ and since $u_k^n\leqslant Y_S< u_{k+1}^n$ and 
$S\le T_{f(u_k^n)}^Z\le T_{u_{k+1}^n}^Y$, we also have $u_k^n \leqslant Y_{T_{f(u_k^n)}^Z}\leqslant u_{k+1}^n$. 
Moreover, recall that $S\ge t_0$, therefore $T_{f(u_k^n)}^Z\ge t_0$. On the other hand since $Z_{S}=f(Y_S)$,
then from time $S$, the process $Z$ will reach the level $f(u_k^n)$ in a time which is at most equal to 
$d_Z^{-1}(f(u_{k}^n)-f(Y_S))$. Hence $T_{f(u_k^n)}^Z\le S+d_Z^{-1}(f(u_{k}^n)-f(Y_S))\le 
t_1+o_1(n)$. We conclude that $\overline{B}_{n,k}^{(1)}$ holds in this subcase. 

In the subcase where $Y$ creeps through the level $u_{k+1}^n$ before time $T_{f(u_k^n)}^Z$, we prove in
the same way that $\overline{B}_{n,k}^{(2)}$ holds.

$(b)$ In the second case, either $Z$ jumps through the level $f(u_k^n)$ and the event $B_{n,k}^{(3)}$ holds 
or $Y$ jumps through the level $u_{k+1}^n$ and the event $B_{n,k}^{(4)}$ holds. Note the both events can occur 
at the same time.

Then the inclusion (\ref{9690}) is proved.

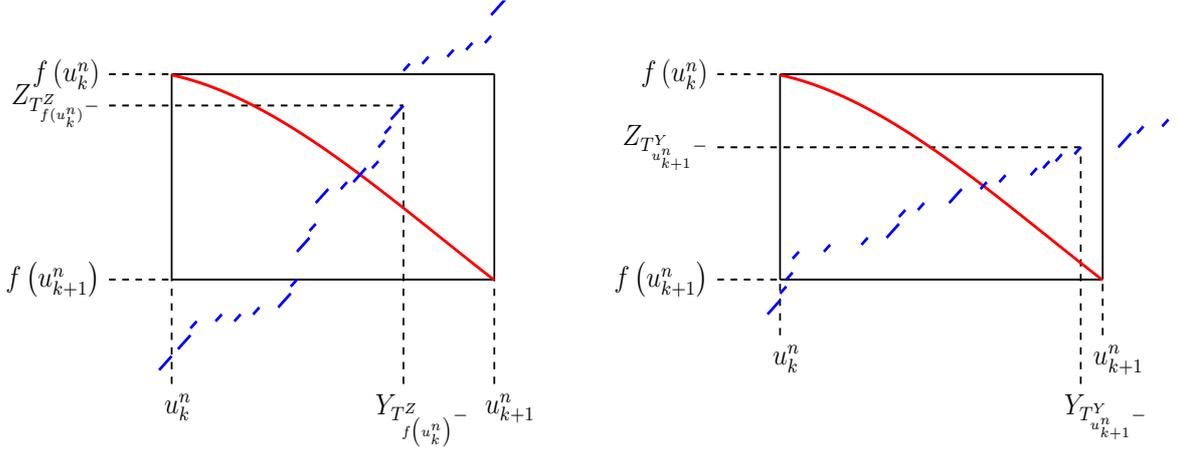
\begin{figure}[!h]

  \begin{minipage}[b]{0.45\linewidth}
   \centering
  \resizebox{10cm}{6cm}{
 \begin{tikzpicture}[xscale=1,yscale=1]
			\newcommand{\xmin}{-2}				
			\newcommand{\xmax}{10}
			\newcommand{\ymin}{-0.5}
			\newcommand{\ymax}{6}
			\clip (\xmin,\ymin) rectangle (\xmax,\ymax);
			\draw[thick](1,4.95)--(1,2);
			\draw[thick](1,2)--(6.15,2);
			\draw[thick](6.15,2)--(6.15,4.95);
			\draw[thick](6.15,4.95)--(1,4.95);
			\draw[dashed, thick](1,0.5)--(1,2);
			\draw[dashed, thick](6.15,0.5)--(6.15,2);
			\draw[dashed, thick](0,2)--(1,2);
			\draw[dashed, thick](0,4.95)--(1,4.95);
			\draw[dashed, thick](0,4.5)--(4.7,4.5);
			\draw[dashed, thick](4.7,0.5)--(4.7,4.5);
			\draw(1.1,0.5)node[below]{$u_k^n$};	
			\draw(6.4,0.5)node[below]{$u_{k+1}^n$};
			\draw(0,2)node[left]{$f\left( u_{k+1}^n \right)$};			
			\draw(0,4.95)node[left]{$f\left( u_{k}^n \right)$};				
			\draw(5,0.5)node[below]{$Y_{T_{f\left( u_{k}^n \right)}^Z-}$};			
			\draw(0,4.5)node[left]{$Z_{T_{f(u_k^n)}^Z-}$};	
			\draw[red, domain=1:6.15, samples=200,very thick] plot (\x, {0.0095*\x*\x*\x-0.151*\x*\x+0.08*\x+5}); 
			\draw[blue, very thick] plot file {coord_B3.txt};
	\end{tikzpicture}}   
  \end{minipage}
\hfill
  \begin{minipage}[b]{0.5\linewidth}
  \centering
 \resizebox{10cm}{6cm}{
\begin{tikzpicture}[xscale=1,yscale=1]
			\newcommand{\xmin}{-2}				
			\newcommand{\xmax}{10}
			\newcommand{\ymin}{-0.5}
			\newcommand{\ymax}{6}
			\clip (\xmin,\ymin) rectangle (\xmax,\ymax);
			\draw[thick](1,4.95)--(1,2);
			\draw[thick](1,2)--(6.15,2);
			\draw[thick](6.15,2)--(6.15,4.95);
			\draw[thick](6.15,4.95)--(1,4.95);
			\draw[dashed, thick](1,1.2)--(1,2);
			\draw[dashed, thick](6.15,1.2)--(6.15,2);
			\draw[dashed, thick](0,2)--(1,2);
			\draw[dashed, thick](0,4.95)--(1,4.95);
			\draw[dashed, thick](0,3.9)--(5.8,3.9);
			\draw[dashed, thick](5.8,0.5)--(5.8,3.9);
			\draw(1.1,1.2)node[below]{$u_k^n$};	
			\draw(6.4,1.2)node[below]{$u_{k+1}^n$};
			\draw(0,2)node[left]{$f\left( u_{k+1}^n \right)$};			
			\draw(0,4.95)node[left]{$f\left( u_{k}^n \right)$};				
			\draw(6.2,0.5)node[below]{$Y_{T_{ u_{k+1}^n}^Y-}$};			
			\draw(0,3.9)node[left]{$Z_{T_{u_{k+1}^n}^Y-}$};	
			\draw[red, domain=1:6.15, samples=200,very thick] plot (\x, {0.0095*\x*\x*\x-0.151*\x*\x+0.08*\x+5}); 
			\draw[blue, very thick] plot file {coord_B4.txt};
	\end{tikzpicture}}    
  \end{minipage}

	\caption{An occurrence of the event $B_{n,k}^{(3)}$ on the left and $B_{n,k}^{(4)}$ on the right}
	\label{evt_B3}
\end{figure}

Let us set $v_1^{(n)}(y,z):=\int_{t_0}^{t_1+o_1(n)}p_t(y,z)\,dt$ and 
$v_2^{(n)}(y,z):=\int_{t_0}^{t_1+o_2(n)}p_t(y,z)\,dt$. 
Then the inclusion (\ref{9690}) implies that 
\begin{eqnarray*}
&&\p(Z_{S}=f(Y_S),\,u_0<Y_{S}<u_1,\,t_0<S<t_1)\\
&\leqslant&\sum_{k=0}^{n-1}\p(Z_{S}=f(Y_S),\,u_k^n\le Y_{S}\le u_{k+1}^{n},\,t_0<S<t_1)\\
&\leqslant&\sum_{k=0}^{n-1}\p(\overline{B}_{n,k}^{(1)})+\sum_{k=0}^{n-1}\p(\overline{B}_{n,k}^{(2)})+\sum_{k=0}^{n-1}
\p(B_{n,k}^{(3)})+\sum_{k=0}^{n-1}\p(B_{n,k}^{(4)})\,,
\end{eqnarray*}
and Lemmas \ref{8333} and \ref{2833} imply that
\begin{eqnarray*}
&&\sum_{k=0}^{n-1}\p(\overline{B}_{n,k}^{(1)})+\sum_{k=0}^{n-1}\p(\overline{B}_{n,k}^{(2)})+
\sum_{k=0}^{n-1}\p(B_{n,k}^{(3)})+\sum_{k=0}^{n-1}\p(B_{n,k}^{(4)})\\
&\leqslant&\sum_{k=0}^{n-1}  d_Z\int_{u_k^n}^{u_{k+1}^n} v_1^{(n)}(u,f(u_k^n))\,du + \sum_{k=0}^{n-1} 
d_Y\int_{f(u_{k+1}^n)}^{f(u_k^n)} v_2^{(n)}(u_{k+1}^n,s)\,ds \\
&&+\sum_{k=0}^{n-1}\frac1n(\varepsilon_1(n)+\varepsilon_2(n))\,.
\end{eqnarray*}
We prove in a similar way as in the previous subsection that the sum of the two first sums in the last
expression tends to $d_Z\int_{u_0}^{u_1}v(u,f(u)) \, du-d_Y\int_{u_0}^{u_1}v(u,f(u))\, df(u)$, when $n$ tend 
to $\infty$. The third sum tends clearly to 0, as $n$ tends to $\infty$, so that inequality (\ref{6293}) is 
proved. 

The upper bound in equality (\ref{7229}) is proved in the same way. This achieves the proof of 
Theorem \ref{9049}.

\subsection{Proof of Proposition \ref{8983}}

Part 1.~is proved as part 1.~of Theorem \ref{9049}, see Subsection \ref{9790}. (Note that $\p(S>0)=1$ since 
$\lim_{t\rightarrow0+}f(t)>0$. In particular, $Y_{S-}$ and $Z_{S-}$ are well defined.)

The proof of part 2.~follows very similar arguments as those used for the proof of the 
upper bound in identities (\ref{7229}) and (\ref{6366}) which is given in Subsection \ref{9791}. We only describe 
it here in broad outline.
Recall the sequence $(\sigma_n)_{n\ge1}$ of subdivisions of $[u_0,u_1]$ defined in (\ref{6392}). Then define the two 
following sequences of events
\begin{eqnarray*}
C_{n,k}^{(1)}&=&\left\lbrace Z_{T_{f(u_{k+1}^n)}^Z}=f(u_{k+1}^n), u_k^n \leqslant 
Y_{T_{f(u_{k+1}^n)}^Z}\leqslant u_{k+1}^n\right\rbrace\\
C_{n,k}^{(2)}&=&\left\lbrace f(u_k^n)\leqslant Z_{T_{f(u_{k+1}^n)}^Z-}<f(u_{k+1}^n),\,
u_k^n\leqslant Y_{T_{f(u_{k+1}^n)}^Z-}<u_{k+1}^n\right\rbrace\,.
\end{eqnarray*}
Since $f$ is non decreasing, the inclusion,
\begin{equation}\label{7453}
C:=\{Z_{S}=f(Y_S),\,u_k^n\leqslant Y_S< u_{k+1}^n\}\subset C_{n,k}^{(1)}\cup C_{n,k}^{(2)}
\end{equation}
holds for all $n$ sufficiently large and $k=0,\dots,n-1$. Indeed, by definition of $S$, there is $\varepsilon>0$
such that the path of $(Y,Z)$ stays above the graph of $f$ on the interval $[S,S+\varepsilon]$. This implies that 
for all $n$ sufficiently large, if $t\ge0$ is such that $Y_{S+t}\in [u_{k}^n,u_{k+1}^n]$ for some $k=0,1,\dots,n-1$, 
then $Z_{S+t}\ge f(S+t)$ so that either the process $Z$ creeps above the level $f(u_{k+1}^n)$, that is 
$C_{n,k}^{(1)}$ holds, or $Z$ jumps above this level, that is $C_{n,k}^{(2)}$ holds. There is no other possibility.

The remainder of the proof is then very similar to the arguments of the end of Subsections \ref{9790} and \ref{9791}, 
the main difference being that the term due to the creeping of $Y$ is now absent. 

\subsection{On assumptions of Theorem \ref{9049}}\label{3331}
The aim of this subsection is to show that Theorem \ref{9049} holds whenever the transition
semigroup of $(Y,Z)$ is absolutely continuous with bounded densities.  Let $d\ge1$ and let ${\mathcal L}$ be any 
possibly killed $\mathbb{R}^d$ valued L\'evy process. 

\begin{proposition}\label{1823} Let $\Psi$ be the characteristic exponent of ${\mathcal L}$, that is 
$\e(e^{i\langle\xi,{\mathcal L}_1\rangle},1<\zeta)=e^{-\Psi(\xi)}$, $\xi\in\mathbb{R}^d$, where $\zeta$ is the 
lifetime of ${\mathcal L}$. Then the following assertions are equivalent:
\begin{itemize}
\item[$(i)$] The transition semigroup of ${\mathcal L}$ is absolutely continuous and there is a version of its 
densities, denoted by $x\longmapsto l_t(x)$, $x \in \R^d$, which are bounded for all $t>0$.
\item[$(ii)$] For all $ t > 0$,
\[e^{-t\Psi(\xi)} \in \mathbf{L}^1(\R^d)\,.\]
\end{itemize}
If these conditions are satisfied, then the function $(t,x)\mapsto l_t(x)$ is continuous on 
$(0,\infty)\times\mathbb{R}^d$. In particular, if all coordinates of ${\mathcal L}$ are subordinators, then 
$l_t(0)=0$, for all $t>0$.
\end{proposition}
\begin{proof}
The proof follows similar arguments as those used in the beginning of Section 2 in \cite{cm}.
Boundness of $l_t$ implies that $l_t\in \mathbf{L}^2(\R^d)$ and 
consequently $e^{-t\Psi(\xi)} \in \mathbf{L}^2(\R^d)$, for all $t>0$, which implies that $e^{-t\Psi(\xi)} \in 
\mathbf{L}^1(\R^d)$, for all $t>0$. Conversely, if $e^{-t\Psi(\xi)} \in \mathbf{L}^1(\R^d)$, for all $ t > 0$, then 
Fourier inversion theorem implies $(i)$. Indeed, by the Riemann-Lebesgue lemma, $l_t \in \mathcal{C}_0(\R^d)$. 
Furthermore, from Fourier inversion theorem we have, for all $(t,x) \in (0,\infty) \times \R^d$, 
\begin{equation}\label{1581}
l_t(x)=\frac{1}{2\pi} \int_{\R^d} e^{-i\langle x,\xi\rangle} e^{-t\Psi(\xi)} d\xi\,.
\end{equation}
Let $(t_0,x_0) \in (0,\infty) \times \R^d$ and $K$ be a compact neighborhood of $(t_0,x_0)$ such that $K\subset  
(0,\infty)\times\R^d$.  Then, there exists $t_1>0$ such that for all $(t,x,\xi) \in K \times \R^d$, 
$|e^{-i\langle x,\xi\rangle}e^{-t\Psi(\xi)}\vert=e^{-t\mbox{\tiny Re}(\Psi(\xi))} 
\leqslant |e^{-t_1\Psi(\xi)}|=e^{-t_1\mbox{\tiny Re}(\Psi(\xi))}$, (recall that $\mbox{Re}(\Psi(\xi))\ge0$). 
Since $e^{-t_1\Psi(\xi)} \in \mathbf{L}^1(\R^d)$, we obtain from (\ref{1581}) and 
Lebesgue's dominated convergence theorem, that $(t,x)\mapsto l_t(x)$ is continuous in $(t_0,x_0)$ and finally 
on $(0,\infty) \times \R^d$.
\end{proof}

\noindent Recall from Subsection \ref{9790} that the absolute continuity of $v(dy,dz)$ is actually enough for 
part 1.~of Theorem \ref{9049} to hold. Therefore, condition $(i)$ (or $(ii)$) of Proposition \ref{1823}
implies part 1.~of Theorem \ref{9049}. Moreover, condition $(i)$ (or $(ii)$) of Proposition \ref{1823} applied 
to the bivariate subordinator $(Y,Z)$ immediately implies the assumption of part 3.~of Theorem \ref{9049}. 
The mapping $(t,y,z)\mapsto p_t(y,z)$ is then locally bounded on $(0,\infty)^3$. Hence the assumption of assertion 
2'.~(stated after Theorem \ref{9049}) is satisfied for all $t_0$, $t_1$ such that $0<t_0<t_1<\infty$ and this implies 
part 2.~of Theorem \ref{9049} by monotone convergence.\\

Let us finally notice that assumptions of Theorem \ref{9049} could be weakened. Indeed, it appears from the 
limits in (\ref{2689}) that the continuity assumption of $v$ could be replaced by an assumption closed to 
the Riemann integrability of this function. (Recall that continuity of $v$ is not required in 
lemmas \ref{2833} and \ref{8333}.)

\section{Proof of Theorem \ref{3636}, Corollary \ref{9445} and Proposition \ref{5838}}\label{proof2}

Recall that $X$ is a real non killed L\'evy process and that we have assumed that $|X|$ is not a subordinator. 

\subsection{Some reminders in fluctuation theory}\label{fluctuation}\label{5399}
Recall that $\overline{X}_{t}=\sup \{X_{s}:0\leq s\leq t\}$ and let 
$\underline{X}_{t}=\inf \{X_{s}:0\leq s\leq t\}$. 
It is well known that both processes $\overline{X}-X$ and $\underline{X}-X$ are strongly Markovian. 
The state 0 is regular for itself, for the process $\overline{X}-X$ (resp.~$\underline{X}-X$) if and only if 
0 is regular for $[0,\infty)$ (resp.~$(-\infty,0]$) for the process $X$. When it is the case, we will simply 
write that $[0,\infty)$ (resp.~$(-\infty,0]$) is regular. 
Similarly, we will write that $(0,\infty)$ (resp.~$(-\infty,0)$)  is regular when 0 is regular for 
$(0,\infty)$ (resp.~$(-\infty,0)$) for the process $X$. Recall also that when $X$ is not a compound 
Poisson process, at least one of the half lines $(-\infty,0)$ or $(0,\infty)$ is regular. Moreover 
compound Poisson processes are the only L\'evy processes for which $[0,\infty)$ and $(-\infty,0]$ are 
regular but not $(0,\infty)$ and $(-\infty,0)$.\\

Let us briefly recall the definition of the local time at 0 of the process $\overline{X}-X$ which we  
denote by $L$. If $[0,\infty)$ is regular, $L$ is a continuous, increasing, additive functional such that 
$L_0=0$, a.s., and the support of the measure $dL_t$ is the set $\overline{\{t:X_t=\overline{X}_t\}}$. 
Moreover $L$ is the unique process, up to a multiplicative constant, satisfying these properties. 
We can normalize it for instance in the following way,
\[\e\left(\int_0^\infty e^{-t}\,dL_t\right)=1\,.\]
When $[0,\infty)$ is not regular, the set $\{t:(X-\overline{X})_t=0\}$ is discrete and the local time at 0 of 
the process $\overline{X}-X$ is defined as follows
\[L_t=\sum_{k=0}^{R_t}{\rm\bf e}^{(k)}\,,\]
where $R_0=0$, for $t>0$, $R_t=\mbox{Card}\{s\in(0,t]:X_s=\overline{X}_s\}$ and ${\rm\bf e}^{(k)}$, 
$k=0,1,\dots$ is a sequence of independent and exponentially distributed random variables with parameter
\[\gamma=\left(1-\e(e^{-\tau^+_0})\right)^{-1}\,,\]
and $\tau_0^+=\inf\{t>0:X_t>0\}$. The choice of the parameter $\gamma$ is consistent with the normalisation
$\e\left(\int_0^\infty e^{-t}\,dL_t\right)=1$ of the regular case. The process $L$ is called the local time 
at the supremum of $X$. The local time at the infimum is defined in the same way with respect to the process 
$\underline{X}-X$ and is denoted by $L^*$. It corresponds to the local time at the supremum of the dual process 
$X^*:=-X$.\\

Let us now recall the definition of the ladder processes. The ladder time processes $\tau$ and $\tau^*$, and the 
ladder height processes $H$ and $H^*$ are the following (possibly killed) subordinators:
\[\tau_t=\inf\{s:L_s>t\}\,,\;\;\tau^*_t=\inf\{s:L_s^*>t\}\,,\;\;H_t=X_{\tau_t}\,,\;\;H^*_t=
-X_{\tau_t^*}\,,\;\;t\ge0\,,\]
where $\tau_t=H_t=+\infty$, for $t\ge L_\infty$ and $\tau_t^*=H_t^*=+\infty$, for $t\ge L^*_\infty$. Recall that 
the processes $(\tau,H)$ and $(\tau^*,H^*)$ are bivariate subordinators. The renewal measures of 
the ladder processes $(\tau,H)$ and $(\tau^*,H^*)$ are the measures on $[0,\infty)^2$ defined by,
\[U(dt,dh)=\int_0^\infty\p(\tau_u\in dt,\,H_u\in dh)\,du\;\;\;\mbox{and}\;\;\;
U^*(dt,dh)=\int_0^\infty\p(\tau^*_u\in dt,\,H^*_u\in dh)\,du\,.\]

The drifts $d_\tau$ and $d_{\tau^*}$ of the subordinators $\tau$ 
and $\tau^*$ satisfy
\begin{equation}\label{delta}
\int_0^t\ind_{\{X_s=\overline{X}_s\}}\,ds=d_{\tau}L_t\,,\;\;\;\int_0^t\ind_{\{X_s=\underline{X}_s\}}\,ds=
d_{\tau^*}L_t^*.
\end{equation}
Recall that $d_{\tau}>0$ if and only if $(-\infty,0)$ is not regular. Similarly, $d_{\tau^*}>0$ if 
and only if 0 is not regular for $(0,\infty)$. Hence if $X$ is not a compound Poisson process, then 
$d_{\tau}d_{\tau^*}=0$ holds since 0 is necessarily regular for at least one of the half lines. If $X$ is 
a compound Poisson process, then $d_{\tau}>0$ and $d_{\tau^*}>0$.\\

Given our normalisation of the local time $L$, the drifts $d_H$ and $d_{H^*}$ of $H$ and $H^*$ satisfy 
$2d_Hd_{H^*}=\sigma^2$, where $\sigma$ represents the Brownian part of $X$ in its L\'evy-Khintchine decomposition
(this follows from the Wiener-Hopf factorisation of the characteristic exponent of $X$, see p.166 in \cite{be}).
When $X$ has bounded variation, we define its drift as the almost sure limit $d:=\lim_{t\rightarrow 0}X_t/t$. 
It was proved in Section 2.~of \cite{mi} that in this case, $d>0$ if and only if $d_H>0$. When $X$ has 
unbounded variation and no Brownian part, Theorem Kaa of \cite{vi} asserts that $d_H>0$ if and only if 
\[\int_0^1\frac{x}{\int_{-x}^0\bar{\bar{\pi}}_1(u)du}\hat{\pi}(x)\,dx<\infty\,,\] 
where $\pi$ is the L\'evy measure of $X$, for $x\in\mathbb{R}$, 
$\hat{\pi}(x)=\int_{\mathbb{R}}\left(\ind_{\{0<x\le u\}}+\ind_{\{u\le x<0\}}\right)d\pi(u)$ and for 
$x\in[-1,0)$, $\bar{\bar{\pi}}_1(x)=\int_{-1}^x\hat{\pi}(u)du$. Note also that if $d_\tau>0$ and $d_H>0$,
then $d_H=dd_\tau$, since $\lim_{t\rightarrow0}\frac{H_t}t=d_H=\lim_{t\rightarrow0}
\frac{X_{\tau_t}}{\tau_t}\frac{\tau_t}{t}=dd_\tau$, a.s.\\

Let us assume for the remainder of this section that $X$ is not a compound Poisson process. The It\^o measure 
$n^*$ of the excursions away from 0 of the process $X-\underline{X}$ is the characteristic measure of the 
Poisson point process 
\[t\mapsto \left\{
\begin{array}{ll}
\{(X-\underline{X})_{\tau_{t-}^*+s},\,0\le s<\tau_t^*-\tau_{t-}^*\}\;\;\;\mbox{if}\;\;\;\tau_{t-}^*<\tau_t^*\\
\delta\;\;\;\mbox{if}\;\;\;\tau_{t-}^*=\tau_t^*\,,
\end{array}\right.\]
where $\delta$ is some isolated point added to the space of excursions. 
We refer to \cite{be}, Chap.~IV, \cite {ky}, Chap.~6 and \cite{do} for a more detailed definition $n^*$. 
The measure $n^*$ is a Markovian measure whose semigroup is this of the killed L\'evy process when it enters in 
the negative half line. More specifically, for $x>0$, let us denote by $\mathbb{Q}_{x}^*$ the law of the process
$(X_t\ind_{\{t<\tau_0^-\}}+\infty\cdot\ind_{\{t\ge\tau_0^-\}},\,t\ge0)$ under $\p_x$,
where $\tau_0^-=\inf\{t>0:X_t<0\}$. That is for $\Lambda \in \mathcal{F}_{t}$,
\begin{equation}
\label{4524}
\mathbb{Q}_{x}^*(\Lambda ,t<\zeta )=\mathbb{P}_{x}(\Lambda ,\,t<\tau_{0}^-)\,.
\end{equation}
Then for all Borel positive functions $f$ and $g$ and for all $s,t>0$,
\begin{equation}
\label{4527}
n^*(f(X_t)g(X_{s+t}),s+t<\zeta)=n^*(f(X_t)\e^{\mathbb{Q}^*}_{X_t}(g(X_s)),s<\zeta)\,,
\end{equation}
where $\e_x^{\mathbb{Q}^*}$ means the expectation under $\mathbb{Q}_x^*$ and $\zeta$ is the lifetime of the 
excursions. Moreover, we define the entrance law $(q_t^*(dx),t>0)$ of the excursion measure $n^*$ as
\begin{equation}\label{4644}
n^*(f(X_t),t<\zeta)=\int_0^\infty f(x)q_t^*(dx)\,.
\end{equation}
The It\^o measure of the excursions away from 0 of the reflected process at its supremum 
$\overline{X}-X=X^*-\underline{X}^*$ is defined in the same way as for $X-\underline{X}$. It will be denoted by 
$n$ and its entrance law by $q_t(dx)$. We define the probability measures $\mathbb{Q}_x$ in the same way as in 
(\ref{4524}) with respect to the dual 
process $X^*=-X$. Let us finally recall that the entrance laws $q_t(dh)$ and $q_t^*(dh)$ are related to the 
renewal measures $U(dt,dh)$ and $U^*(dt,dh)$ of $(\tau,H)$ and $(\tau^*,H^*)$ as follows,
\begin{eqnarray}
d_{\tau^*}\delta_{\{(0,0)\}}(dt,dh)+q_t^*(dh)\,dt&=&U(dt,dh)\label{1511}\\
d_{\tau}\delta_{\{(0,0)\}}(dt,dh)+q_t(dh)\,dt&=&U^*(dt,dh).\label{1512}
\end{eqnarray}
These identities between measures on $[0,\infty)^2$ can be found in Lemma 1 of \cite{ch}. In particular, 
under the assumption of Theorem \ref{3636}, that is when $U(dt,dh)$ has a density on $(0,\infty)^2$, 
this density is the same as this of $q_t^*(dh)\,dt$, which explains the notation $q_t^*(h)$, $t,h\in(0,\infty)$
in this theorem. 

\subsection{Proof of Theorem \ref{3636}}\label{8277}

Part 1.~of Theorem \ref{3636} is a direct application of Lemma \ref{5698}. It suffices to consider 
$\mathcal{L}_t=(X_t,t)$ and $G(u,v)=u-f(v)$ in this lemma. (Note that $\p(T_f>0)=1$ follows from the fact 
that $f$ is positive. In particular, $X_{T_f-}$ is well defined on $\{T_f<\infty\}$.)

Then part 2.~is a consequence of Lemma \ref{5698} and Theorem \ref{9049}.
Recall first from Subsection \ref{5399} that $(\tau,H)$ is a bivariate subordinator. If the time $T_f=\inf\{t:X_t>f(t)\}$ 
satisfies $X_{T_f}=\overline{X}_{T_f}=f({T_f})$, then it is a zero of $\overline{X}-X$ and hence there is $S'$ 
such that $T_f=\tau_{S'}$ or $T_f=\tau_{S'-}$. Therefore $H_{S'}=f(\tau_{S'})$ or 
${X}_{\tau_{S'-}}=\overline{X}_{\tau_{S'-}}=f(\tau_{S'-})$ but we know, applying again Lemma \ref{5698} to the function 
$G$ as above, that in the latter case, $X$ is necessarily a.s.~continuous at time $\tau_{S'-}$. In particular, 
$H_{S'-}={X}_{\tau_{S'-}-}=f(\tau_{S'-})$. On the other hand, applying again Lemma \ref{5698}, it follows that if 
$H_{S'-}=f(\tau_{S'-})$, then $(\tau,H)$ is necessarily a.s.~continuous at time $S'$, so that 
$\tau_{S'}=\tau_{S'-}$. This proves that $H_{S'}=f(\tau_{S'})$, a.s. Moreover, since the subordinator $H$ is not lattice, 
the time $S'$ is necessarily unique, so that $S'=S=\inf\{t\ge0:H_t>f(\tau_t)\}$. This proves in particular that 
$\tau_S=T_f$.

Conversely assume that $S$ satisfies $H_{S}=X_{\tau_{S}}=\overline{X}_{\tau_{S}}=f(\tau_{S})$. Since $X$ is not lattice,
and $f$ is non increasing, a time $T$ satisfying $X_{T}=\overline{X}_{T}=f(T)$ is necessarily unique. Moreover, when $X$ 
is not a compound Poisson process, $\tau_S=\inf\{t:X_t>X_{\tau_S}\}$. Hence, since $f$ is non increasing, 
$\tau_{S}=\inf\{t:\overline{X}_t>f(t)\}=T_f$, so that $X_{T_f}=\overline{X}_{T_f}=f({T_f})$ and we have proved the 
identity
\[\{H_S=f(\tau_S),\,u_0<\tau_S<u_1\}=\{X_{T_f}=\overline{X}_{T_f}=f({T_f}),\,u_0<T_f<u_1\}\]
which allows us to conclude thanks to part 2.~of Theorem \ref{9049}.

\subsection{On assumptions of Theorem \ref{3636}}\label{preliminaries}

It is assumed in part 2.~of Theorem \ref{3636} that the renewal measure of the ladder process $(\tau,H)$ has  
continuous densities $(t,x)\mapsto q_t^*(x)$ on $(0,\infty)^2$. The next proposition provides conditions for this to hold.  
Moreover, when the later are satisfied, $q_t^*$ is positive on $(0,\infty)$, for all $t>0$. This ensures the positivity 
of the creeping probability (\ref{8225}) whenever $d_H>0$ or $d_\tau>0$ and $f(t_1)<f(t_0)$.   

\begin{proposition}\label{1753} 
Assume that $X$ is a non killed real L\'evy process which satisfies the equivalent conditions $(i)$ and $(ii)$ of 
Proposition $\ref{1823}$, then 
\begin{itemize}
\item[$1.$] the semigroup $q_t^*(x,dy)$ of the killed process $(X,\mathbb{Q}^*_x)$ defined in $(\ref{4524})$ 
is absolutely continuous and its densities $(t,x,y)\mapsto q_t^*(x,y)$ are continuous in 
$(0,\infty)^2\times[0,\infty)$.
\item[$2.$] Assume moreover that for all $c\ge0$, the process $(|X_t-ct|,\,t\ge0)$ is not a subordinator. 
Then the entrance law, $q_t^*(dx)$ is absolutely continuous on $[0,\infty)$ for all $t>0$ and there is 
a version $q_t^*(x)$, $x\ge0$, $t>0$ of its densities such that the function $(t,x)\mapsto q^*_t(x)$ is 
continuous on $(0,\infty)\times[0,\infty)$. Moreover, for all $t>0$, $q^*_t$ is positive on $(0,\infty)$. 

Denote by $h$ the renewal function of $H$, that is $h(x)=\int_0^\infty\p(H_t\le x)\,dt$, 
then $\lim_{y\rightarrow0+}q_t^*(x,y)=h(0)q_t(x)$.  Moreover for all $t>0$, 
$\lim_{x\rightarrow0+}q_t^*(x)=h(0)p_t(0)/t$. 
We then set $q_t^*(x,0)=h(0)q_t(x)$ and $q_t^*(0)=h(0)p_t(0)/t$, for all $t,x>0$.
\end{itemize}
\end{proposition}
\begin{proof}
The first assertion is a slight reinforcement of Lemma 2 in \cite{ur} and our proof follows the same arguments.
Recall first the definition (\ref{4524}) of the killed process $(X,\mathbb{Q}^*_x)$ and let us denote by 
$q_t^*(x,dy)$ its semigroup. Then from Lemma 2 in \cite{ur}, under our assumption and thanks to Proposition 
\ref{1823}, this semigroup admits densities $q_t^*(x,y)$ which satisfy for all $t>0$, $x>0$, $y\ge0$, 
\[q_t^*(x,y)=\p_{x,y}^t(\Lambda_t)p_t(y-x)\,,\]
where $\p_{x,y}^t$ is the law of the bridge of $(X,\p)$ starting from $x$ and ending in $y$ at time $t$ and 
$\Lambda_t=\{\underline{X}_t>0\}\cup\{X_s>0,\,0\le s<t,\,X_t=0\}$. 
Here, in order to define $q_t^*(x,0)$ for all $t>0$ and $x>0$, we have included $y=0$ since the right member is 
then well defined. From Corollary 1 of \cite{cu}, the measures $\p_{x,y}^t$ are weakly continuous in 
$(t,x,y)\in(0,\infty)\times\mathbb{R}^2$. Then from the same arguments as in Lemma 2 of \cite{ur}, these measures 
satisfy $\p_{x,y}^t(\partial\Lambda_t)=0$ for all $t,x>0$ and $y\ge0$. 
Hence from portemanteau theorem and Proposition \ref{1823}, the mapping 
$(t,x,y)\mapsto\p_{x,y}^t(\Lambda_t)p_t(y-x)$ is continuous in $(0,\infty)^2\times[0,\infty)$.

In order to prove $2.$, recall from part 3.~of Lemma 1 in \cite{ch} that $q_t^*(dx)$ is absolutely continuous 
under our assumptions. Let us denote by $q_t^*(x)$ the corresponding densities and recall the Chapman-Kolmogorov 
equation,
\begin{equation}\label{6721}
q_{t+s}^*(y)=\int_0^\infty q_s^*(x)q_t^*(x,y)\,dx\,,\;\;\;s,t,y>0\,,
\end{equation}
which can be derived for instance from (\ref{4527}). From 
Proposition 1 in \cite{cm}, for all $x>0$ and $t>0$, $\lim_{y\rightarrow0+}q_t^*(x,y)=h(0)q_t(x)$. Let us set 
$q_t^*(x,0)=h(0)q_t(x)$. Let $\Psi$ be the characteristic exponent of $X$, that is 
$\e(e^{i\xi X_1})=e^{-\Psi(\xi)}$. From Lemma 2 in \cite{ur}, $q_t^*(x,y)\le p_t(y-x)$, for all 
$t>0$ and $x,y\in(0,\infty)^2$, so that, 
\begin{equation}\label{4722}
q_t^*(x,y)\le p_t(y-x)=\frac{1}{2\pi} \int_{\R} e^{-i(y-x) \xi} e^{-t\Psi(\xi)} d\xi\le
\frac{1}{2\pi}\int_{\R} |e^{-t\Psi(\xi)}| d\xi\,,
\end{equation} 
and this bound can be extended to $y=0$. From the same arguments as in the proof of Proposition \ref{1823}, 
the mapping $t\mapsto \int_{\R} |e^{-t\Psi(\xi)}| d\xi$ is continuous on $(0,\infty)$. Let 
$(t_0,y_0)\in(0,\infty)\times[0,\infty)$, then from (\ref{4722}) and the continuity of 
$t\mapsto \int_{\R} |e^{-t\Psi(\xi)}| d\xi$, there is $\varepsilon>0$ and $C_\varepsilon>0$ 
such that $q_t^*(x,y)\le C_\varepsilon$, for all $x>0$, $y\ge0$ and $t\in(t_0-\varepsilon,t_0+\varepsilon)$.
Moreover, from (\ref{4527}) $\int_0^\infty q_s^*(x)\,dx=n^*(s<\zeta)<\infty$, 
hence, (\ref{6721}) can be extended to $y=0$ and we can set $q_t^*(0)=\lim_{y\rightarrow0+}q_t^*(y)$, $t>0$. 
Moreover, continuity of $(t,y)\mapsto q_t^*(y)$ at $(t_0,y_0)$ follows from (\ref{6721}), continuity of 
$(t,x,y)\mapsto q_t^*(x,y)$ in $(0,\infty)^2\times[0,\infty)$ proved in assertion 1.~and 
Lebesgue theorem of dominated convergence. 

The positivity of $q_t^*$ for all $t>0$ follows from Lemma 2 in \cite{cm} and the two last assertions in $2.$ 
follow from Proposition 1 in \cite{cm}.
\end{proof}

\subsection{Proof of Corollary \ref{9445}}
  
As recalled in Subsection \ref{fluctuation}, since the process $Y_t:=X_t+at$, $t\ge0$ has bounded variation and 
positive drift, its upward ladder height process also has positive drift. Moreover, by assumption, for all $t>0$ the 
distribution $\p(X_t\in dx)$ is absolutely continuous with a bounded density, and for all $c\ge0$, the process 
$(|X_t-ct|,\,t\ge0)$ is not a subordinator. Then these conditions are clearly also satisfied by the process $Y$. 
Let us denote by $q_t^{*,a}(x)$ the density of the renewal measure of the ladder process associated 
to $Y$. Then from part 2.~of Proposition \ref{1753}, $(t,x)\mapsto q_t^{*,a}(x)$ is continuous on $(0,\infty)^2$. 
Moreover, $q_t^{*,a}$ is positive on $(0,\infty)$ for all $t>0$. These properties 
imply, from part 2.~of Theorem \ref{3636}, that the process $Y$ creeps at its supremum through the function 
$f(t)+at$ over the interval $(t_0,t_1)$, so that $X$ creeps through the function $f(t)$ over the interval 
$(t_0,t_1)$.

\subsection{Proof of Proposition \ref{5838}}

Proposition \ref{5838} is a consequence of Proposition \ref{8983} and its proof follows the same arguments as those used
for the proof of Theorem \ref{3636} in Subsection \ref{8277}. The details are omitted.

\section{Proof of Theorem \ref{3737}}\label{proof3}

The proof is a direct application of the extension of Tanaka's construction of random walks conditioned
to stay positive obtained in \cite{do1}. Let us first state this transformation. Recall that 
$\overline{X}_t=\sup_{s\le t}X_s$ and let us set for $t>0$,
\[g_t=\sup\{s<t:X_s=\overline{X}_s\}\;\;\;\mbox{and}\;\;\;d_t=\inf\{s>t:X_s=\overline{X}_s\}\,.\]
Then the process defined by
\begin{eqnarray*}
W_0&=&0\\
W_t&=&\overline{X}_{d_t}+\ind_{\{d_t>g_t\}}(\overline{X}-X)_{(d_t+g_t-t)-}\,,\;\;\;t>0\,,
\end{eqnarray*}
has the law of $X$ conditioned to stay positive, $X^\uparrow$. According to this construction, the process 
$W$ is obtained by time reversing of each excursion of the reflected process $\overline{X}-X$ and then by gluing 
them back together.

The assertion $0<\sigma_f<\infty$, a.s. is straightforward since $\lim_{t\rightarrow0+}f(t)>0$ and 
$\lim_{t\rightarrow\infty}X_t^\uparrow=\infty$, almost surely. Then we derive from the above construction that 
almost surely, 
\[\overline{X}_t=\underline{\underline{W}}_t\,,\;\;\;\mbox{for all $t\ge0$.}\]
Moreover, the sets $\{t\ge0:X_t=\overline{X}_t\}$  and 
$\{t\ge0:W_t=\underline{\underline{W}}_t\}$ are identical. It follows that 
$\overline{X}_{T_f}=\underline{\underline{W}}_{T_f}$, almost surely and that $X_{T_f}=\overline{X}_{T_f}$ if and only 
if $W_{T_f}=\underline{\underline{W}}_{T_f}$. Then it remains to observe that on the set 
$\{X_{T_f}=\overline{X}_{T_f}\}=\{W_{T_f}=\underline{\underline{W}}_{T_f}\}$, $T_f=\sup\{t:W_t\le f(t)\}=\sigma_f$ and 
this proves parts 1.~and 2.~of Theorem \ref{3737}.

\end{document}